\documentclass{article}%
\usepackage{mitpress}
\usepackage{amsmath}
\usepackage{graphicx}
\usepackage{amsfonts}
\usepackage{amssymb}%
\providecommand{\U}[1]{\protect\rule{.1in}{.1in}}
\providecommand{\U}[1]{\protect\rule{.1in}{.1in}}
\providecommand{\U}[1]{\protect\rule{.1in}{.1in}}
\providecommand{\U}[1]{\protect\rule{.1in}{.1in}}
\newtheorem{theorem}{Theorem}
\newtheorem{acknowledgement}[theorem]{Acknowledgement}

\newtheorem{corollary}[theorem]{Corollary}

\newtheorem{definition}[theorem]{Definition}

\newtheorem{remark}[theorem]{Remark}

\newenvironment{proof}[1][Proof]{\textbf{#1.} }{\ \rule{0.5em}{0.5em}}
\newdimen\dummy
\dummy=\oddsidemargin
\begin{document}

\title{A new cubature formula for weight functions on the disc, with error estimates}
\author{O. Kounchev and H. Render}
\maketitle

\begin{abstract}
We introduce a new type of cubature formula for the evaluation of an integral
over the disk with respect to a weight function$.$ The method is based on an
analysis of the Fourier series of the weight function and a reduction of the
bivariate integral into an infinite sum of univariate integrals. Several
experimental results show that the accuracy of the method is superior to
standard cubature formula on the disk. Error estimates provide the theoretical
basis for the good performance of the new algorithm.

\end{abstract}

\section{Introduction}

Recently, methods for the numerical evaluation of integrals of the form
\begin{equation}
I_{1}\left(  g\right)  =\int_{D}g\left(  x\right)  dx=\int_{0}^{2\pi}\int
_{0}^{R}g\left(  r\cos\varphi,r\sin\varphi\right)  rdrd\varphi\label{I1}%
\end{equation}
on the disc $D_{R}$ of radius $R$ in the plane $\mathbb{R}^{2}$ have received
increased attention in the framework of the meshless local Petrov-Galerkin
(MLPG) method, see \cite{DeBa00}, \cite{DeBa01}, \cite{MaPi10}, \cite{MFPG07}.
Numerical experiments in \cite{DeBa01} have given evidence that classical
rules like the \emph{piecewise midpoint quadrature rule}, or the \emph{rule of
Peirce} (for definitions see below (\ref{Rule1}) and (\ref{modPeirce})) are
superior to the \emph{Gauss-Legendre product rule} which is very popular in
the MLPG literature.

In the present paper we study new methods for the numerical evaluation of
integrals of the type%
\begin{equation}
I_{w}\left(  f\right)  =\int_{D_{R}}f\left(  x\right)  w\left(  x\right)  dx
\label{eqint}%
\end{equation}
where $w\left(  x\right)  $ is a (not necessarily non-negative) weight
function on the disc $D_{R}$ in the plane $\mathbb{R}^{2}.$ The introduction
of a weight function is an important concept in numerical integration: the
integrand $g\left(  x\right)  $ is decomposed into a product $f\left(
x\right)  w\left(  x\right)  $ where the factor $f\left(  x\right)  $ is a
function well-approximable by polynomials (see section $3.7$ in \cite{stoer})
and $w\left(  x\right)  $ is a function of limited smoothness or with a
singularity. Using the specific properties of the weight function $w\left(
x\right)  $ one aims to achieve a cubature formula for the integration of the
function $f$ with respect to $w\left(  x\right)  dx$ which should be more
accurate than using directly a cubature formulae for $g=f\cdot w$ like in
(\ref{I1}).

The main concept underlying our method is to expand the weight function
$w\left(  x\right)  =w\left(  re^{i\varphi}\right)  $ into a Fourier series
(\ref{FourierSeries}) and rely upon a similar expansion for the
polynomial-like function $f\left(  x\right)  $, called the Almansi expansion,
see (\ref{Almansi}). The reader will find an explicit description of this
construction below, after all necessary notations and tools are introduced.
Illustrating examples in this paper are the weight functions
\begin{equation}
w^{\left(  1\right)  }\left(  x,y\right)  =\frac{1+x}{\sqrt{x^{2}+y^{2}}}%
\quad\text{ and }\quad w^{\left(  2\right)  }\left(  x,y\right)  =\left\vert
y\right\vert \ \label{eqweight}%
\end{equation}
where the first weight function has a singularity in $0$ and the second weight
function is continuous on the closed ball but is not differentiable on the
line $y=0$ in the interior of the disk. We shall show by numerical
experiments, and by theoretical considerations as well, that our method
provides in these cases results which are better than the above-mentioned
methods for the approximation of the integral for the function $g\left(
x\right)  =f\left(  x\right)  w\left(  x\right)  .$

Let us now give a detailed introduction to the main topic of the present
paper, the numerical evaluation of the integral with respect to a weight
function $w.$ A central role in our approach plays the Fourier series of $w,$
given by
\begin{equation}
w\left(  re^{i\varphi}\right)  =w\left(  r\cos\varphi,r\sin\varphi\right)  :=%
{\displaystyle\sum_{k=0}^{\infty}}
{\displaystyle\sum_{\ell=1}^{a_{k}}}
w_{\left(  k,\ell\right)  }\left(  r\right)  Y_{\left(  k,\ell\right)
}\left(  \varphi\right)  . \label{FourierSeries}%
\end{equation}
Here we use a notation for the Fourier series which is more convenient in our
context, and which is well known from the theory of spherical harmonics (these
convenient notations will be important also for further multivariate
generalizations, as in \cite{Dolo13}): we work with the
\emph{orthonormalization} of the harmonics $\cos k\varphi$ and $\sin k\varphi
$, defined by
\begin{align}
Y_{\left(  0,1\right)  }\left(  \varphi\right)   &  =1/\sqrt{2\pi
}\label{sphericalHarmonics1}\\
Y_{\left(  k,1\right)  }\left(  \varphi\right)   &  =\frac{1}{\sqrt{\pi}}\cos
k\varphi\quad\text{ and }\quad Y_{\left(  k,2\right)  }\left(  \varphi\right)
=\frac{1}{\sqrt{\pi}}\sin k\varphi. \label{sphericalHarmonics2}%
\end{align}
for integers $k\geq1.$ Then $Y_{\left(  k,\ell\right)  }$ is an
\emph{orthonormal} system for $k\geq0,\ell=1,..,a_{k}$, where $a_{k}=2$ for
$k\geq1$, and $a_{0}=1.$ The $\left(  k,\ell\right)  $-th Fourier coefficient
of a complex-valued continuous function $f\left(  re^{i\varphi}\right)  $ is
\begin{equation}
f_{\left(  k,\ell\right)  }\left(  r\right)  :=\int_{0}^{2\pi}f\left(
re^{i\varphi}\right)  Y_{\left(  k,\ell\right)  }\left(  \varphi\right)
d\varphi\qquad\text{ for }k\geq0,\ell=1,..,a_{k} \label{fkl}%
\end{equation}
and the \emph{corresponding Fourier series} of $f$ is
\begin{equation}
f\left(  re^{i\varphi}\right)  =f\left(  r\cos\varphi,r\sin\varphi\right)  :=%
{\displaystyle\sum_{k=0}^{\infty}}
{\displaystyle\sum_{\ell=1}^{a_{k}}}
f_{\left(  k,\ell\right)  }\left(  r\right)  Y_{\left(  k,\ell\right)
}\left(  \varphi\right)  . \label{FourierSeriesGeneral}%
\end{equation}
Let us recall that the Fourier series of a polynomial $p\left(  x,y\right)  $
is of a very special form: there exist polynomials $\widetilde{p}_{\left(
k,\ell\right)  }$ and a number $N\leq\deg p\left(  x,y\right)  $ such that
\begin{equation}
p\left(  x,y\right)  =p\left(  r\cos\varphi,r\sin\varphi\right)  =%
{\displaystyle\sum_{k=0}^{N}}
{\displaystyle\sum_{\ell=1}^{a_{k}}}
\widetilde{p}_{\left(  k,\ell\right)  }\left(  r^{2}\right)  r^{k}Y_{\left(
k,\ell\right)  }\left(  \varphi\right)  ; \label{Almansi}%
\end{equation}
the representation (\ref{Almansi}) is called Gauss decomposition or Almansi
expansion of a polynomial $p$. Hence, the Fourier coefficient $p_{\left(
k,\ell\right)  }\left(  r\right)  $ of a polynomial $p\left(  x,y\right)  $ is
of the form
\begin{equation}
p_{\left(  k,\ell\right)  }\left(  r\right)  =\widetilde{p}_{\left(
k,\ell\right)  }\left(  r^{2}\right)  r^{k}. \label{Almansi2}%
\end{equation}
Moreover the degrees of all $\widetilde{p}_{\left(  k,\ell\right)  }$ are
bounded by $N-1$ if and only if the polynomial $p\left(  x,y\right)  $ is
\textbf{polyharmonic} of order $N,$ i.e. if $\Delta^{N}p\left(  x,y\right)
=0$, where $\Delta^{N}$ is the $N$-th iterate of the Laplace operator
$\Delta.$ These results are given a thorough treatment in see \cite{okbook},
\cite{kounchevRenderArkiv}, \cite{sobolev}.

We consider now the integral (\ref{eqint}), which after introducing polar
coordinates, becomes%

\[
I_{w}\left(  f\right)  =\int_{0}^{2\pi}\int_{0}^{R}f\left(  r\cos\varphi
,r\sin\varphi\right)  \cdot w\left(  r\cos\varphi,r\sin\varphi\right)  \cdot
rdrd\varphi.
\]
We replace the weight function $w\left(  re^{i\varphi}\right)  $ by its
Fourier series, and after interchanging summation and integration we obtain
\begin{equation}
I_{w}\left(  f\right)  =\int_{0}^{2\pi}\int_{0}^{R}f\left(  re^{i\varphi
}\right)  w\left(  re^{i\varphi}\right)  rdrd\varphi=%
{\displaystyle\sum_{k=0}^{\infty}}
{\displaystyle\sum_{\ell=1}^{a_{k}}}
\int_{0}^{R}f_{\left(  k,\ell\right)  }\left(  r\right)  w_{\left(
k,\ell\right)  }\left(  r\right)  rdr. \label{eqcentral2}%
\end{equation}
Note that for the constant weight function $w\left(  x,y\right)  =1=\sqrt
{2\pi}Y_{\left(  0,1\right)  }$ one obtains simply
\begin{equation}
I_{1}\left(  f\right)  =\int_{0}^{2\pi}\int_{0}^{R}f\left(  r\cos\varphi
,r\sin\varphi\right)  rdrd\varphi=\sqrt{2\pi}\int_{0}^{R}f_{\left(
0,1\right)  }\left(  r\right)  rdr. \label{eqintegralpolar}%
\end{equation}
Formula (\ref{eqcentral2}) is central to our approach since it reduces the
integration of the bivariate function $f\left(  re^{i\varphi}\right)  $ with
respect to a weight function $w\left(  re^{i\varphi}\right)  $ to the
calculation of an infinite family of univariate integrals with weight
functions $w_{\left(  k,\ell\right)  }\left(  r\right)  $.\emph{ }

Here we come to the most crucial point of our approach: as we assume that the
function $f$ is well-approximable by a polynomial $p$, we know that
$I_{w}\left(  f\right)  $ is close to $I_{w}\left(  p\right)  .$ By the
Almansi formula (\ref{Almansi}) and (\ref{Almansi2}), the one-dimensional
integrals in (\ref{eqcentral2}) for computing $I_{w}\left(  p\right)  $ are
equal to
\[
\int_{0}^{R}p_{k,\ell}\left(  r\right)  w_{\left(  k,\ell\right)  }\left(
r\right)  rdr=\int_{0}^{R}\widetilde{p}_{k,\ell}\left(  r^{2}\right)
r^{k}w_{\left(  k,\ell\right)  }\left(  r\right)  rdr.
\]
After a change of the variable $\rho=r^{2}$ it becomes
\begin{equation}
\frac{1}{2}\int_{0}^{R^{2}}p\left(  \rho\right)  \cdot\rho^{k/2}w_{\left(
k,\ell\right)  }\left(  \sqrt{\rho}\right)  d\rho. \label{defTkl}%
\end{equation}
Now we want to employ the $N$-point Gauss-Jacobi quadrature to the last
integral, with measure $\rho^{k/2}w_{\left(  k,\ell\right)  }\left(
\sqrt{\rho}\right)  .$ For this reason we need the following assumption which
will be made throughout the entire paper and which is called the
pseudo-definiteness of the weight function:

\bigskip

\textbf{General Assumption: } \emph{Each Fourier coefficient }$w_{\left(
k,\ell\right)  }\left(  r\right)  $\emph{\ of the weight function }$w$\emph{
does not change the sign over the interval }$\left(  0,R\right)  ,$ \emph{ and
it is integrable and continuous over }$\left(  0,R\right)  .$

\bigskip

Due to our assumption we infer the existence of the $N$-point Gauss-Jacobi
quadrature with nodes and coefficients (which are either all positive or all
negative)
\begin{gather}
t_{1,\left(  k,\ell\right)  }<...<t_{N,\left(  k,\ell\right)  }\label{nodes}\\
\lambda_{1,\left(  k,\ell\right)  },...,\lambda_{N,\left(  k,\ell\right)  }.
\label{coefs}%
\end{gather}
Due to the exactness of the Gauss-Jacobi quadrature for any integer $0\leq
s\leq2N-1$ we obtain the equalities
\begin{equation}%
{\displaystyle\sum_{j=1}^{N}}
\lambda_{j,\left(  k,\ell\right)  }\cdot t_{j,\left(  k,\ell\right)  }%
^{s}=\frac{1}{2}\int_{0}^{R^{2}}\rho^{s}\rho^{k/2}w_{\left(  k,\ell\right)
}\left(  \sqrt{\rho}\right)  d\rho=\int_{0}^{R}r^{2s}r^{k}w_{\left(
k,\ell\right)  }\left(  r\right)  rdr. \label{eqexactkl}%
\end{equation}
Hence, for a polynomial $f,$ for which $\deg f_{\left(  k,\ell\right)  }%
\leq2N-1$ for all $\left(  k,\ell\right)  $, by using the Gauss-Jacobi
quadrature (\ref{eqexactkl}), the integral (\ref{eqcentral2}) becomes
\begin{align}
I_{w}\left(  f\right)   &  =%
{\displaystyle\sum_{k=0}^{\infty}}
{\displaystyle\sum_{\ell=1}^{a_{k}}}
\int_{0}^{R}f_{\left(  k,\ell\right)  }\left(  r\right)  w_{\left(
k,\ell\right)  }\left(  r\right)  rdr\label{Iwf2}\\
&  =\frac{1}{2}%
{\displaystyle\sum_{k=0}^{\infty}}
{\displaystyle\sum_{\ell=1}^{a_{k}}}
\int_{0}^{R^{2}}f_{\left(  k,\ell\right)  }\left(  \sqrt{\rho}\right)
\rho^{-k/2}\left\{  \rho^{k/2}w_{\left(  k,\ell\right)  }\left(  \sqrt{\rho
}\right)  \right\}  d\rho\nonumber\\
&  =I_{N}^{\text{poly}}\left(  f\right) \nonumber
\end{align}
where we have put
\begin{equation}
I_{N}^{\text{poly}}\left(  f\right)  :=\frac{1}{2}%
{\displaystyle\sum_{k=0}^{\infty}}
{\displaystyle\sum_{\ell=1}^{a_{k}}}
{\displaystyle\sum_{j=1}^{N}}
\lambda_{j,\left(  k,\ell\right)  }\cdot t_{j,\left(  k,\ell\right)  }%
^{-\frac{k}{2}}\cdot f_{\left(  k,\ell\right)  }\left(  \sqrt{t_{j,\left(
k,\ell\right)  }}\right)  . \label{eqIpoly}%
\end{equation}

In \cite{Dolo13}, we defined $I_{N}^{\text{poly}}\left(  f\right)  $ as
\emph{polyharmonic cubature of degree} $N$ in arbitrary space dimension. The
reason for the name is the fact that $I_{N}^{\text{poly}}$ is exact on the
space of all polynomials of polyharmonic order $\leq2N,$ i.e. for each
polynomial $f$ such that $\Delta^{2N}f=0,$ or more general, on the space of
\emph{smooth functions }satisfying the polyharmonic equation $\Delta^{2N}f=0$

In previous work \cite{kounchevRenderArkiv}, \cite{kounchevRenderKleinDirac},
\cite{kounchevRenderArxiv}, we have given a motivation and a detailed analysis
of the polyharmonic cubature of degree $2N$ in the framework of the
Polyharmonic Paradigm, explaining the natural appearance of the factor
$t_{j,\left(  k,\ell\right)  }^{-\frac{k}{2}}$ in our formulae which is
related to the Gauss-Almansi decomposition of a polynomial.

As the values $f_{\left(  k,\ell\right)  }\left(  \sqrt{t_{j,\left(
k,\ell\right)  }}\right)  $ are Fourier coefficients, they can be approximated
by means of Discrete Cosine/Sine transform of the function $f$, by which we
mean the expression
\begin{equation}
f_{\left(  k,\ell\right)  }^{\left(  M\right)  }\left(  r\right)  :=\frac
{2\pi}{M}%
{\displaystyle\sum_{s=1}^{M}}
f\left(  re^{i\frac{2\pi s}{M}}\right)  Y_{\left(  k,\ell\right)  }\left(
\frac{2\pi s}{M}\right)  \label{DFT}%
\end{equation}

The main contribution of the present paper is the \textbf{Discrete
Polyharmonic Cubature with parameters} $\left(  N,M,K\right)  ,$ defined for
integers $N\geq1$, $M\geq1,$ and $K\geq0,$ by putting
\begin{align}
I_{\left(  N,M,K\right)  }^{\text{poly}}\left(  f\right)   &  :=\frac{1}{2}%
{\displaystyle\sum_{k=0}^{K}}
{\displaystyle\sum_{\ell=1}^{a_{k}}}
{\displaystyle\sum_{j=1}^{N}}
\lambda_{j,\left(  k,\ell\right)  }\cdot t_{j,\left(  k,\ell\right)  }%
^{-\frac{k}{2}}\cdot f_{\left(  k,\ell\right)  }^{\left(  M\right)  }\left(
\sqrt{t_{j,\left(  k,\ell\right)  }}\right) \label{DiscrPolyhCubature}\\
&  =\frac{\pi}{M}%
{\displaystyle\sum_{k=0}^{K}}
{\displaystyle\sum_{\ell=1}^{a_{k}}}
{\displaystyle\sum_{j=1}^{N}}
{\displaystyle\sum_{s=1}^{M}}
\lambda_{j,\left(  k,\ell\right)  }\cdot t_{j,\left(  k,\ell\right)  }%
^{-\frac{k}{2}}\cdot Y_{\left(  k,\ell\right)  }\left(  \frac{2\pi s}%
{M}\right)  \cdot f\left(  \sqrt{t_{j,\left(  k,\ell\right)  }}e^{i\frac{2\pi
s}{M}}\right)  . \label{DiscrPolyhCubature+}%
\end{align}

We see that unlike (\ref{eqIpoly}) formula (\ref{DiscrPolyhCubature+}) for
$I_{\left(  N,M,K\right)  }^{\text{poly}}\left(  f\right)  $ is indeed a
cubature formula in the usual sense of the word. Its coefficients (weights)
$\left\{  \lambda_{j,\left(  k,\ell\right)  }\cdot t_{j,\left(  k,\ell\right)
}^{-\frac{k}{2}}\cdot Y_{\left(  k,\ell\right)  }\left(  \frac{2\pi s}%
{M}\right)  \right\}  $ have varying signs but they satisfy the following
\emph{remarkable} inequality
\begin{equation}
\frac{\pi}{M}%
{\displaystyle\sum_{k=0}^{K}}
{\displaystyle\sum_{\ell=1}^{a_{k}}}
{\displaystyle\sum_{j=1}^{N}}
{\displaystyle\sum_{s=1}^{M}}
\left\vert \lambda_{j,\left(  k,\ell\right)  }\cdot t_{j,\left(
k,\ell\right)  }^{-\frac{k}{2}}\cdot Y_{\left(  k,\ell\right)  }\left(
\frac{2\pi s}{M}\right)  \right\vert \leq\sqrt{\pi}\left\Vert w\right\Vert ,
\label{wRemarkable}%
\end{equation}
where it is assumed that the weight function $w$ satisfies the so-called
summability condition
\begin{equation}
\left\Vert w\right\Vert :=\sum_{k=0}^{\infty}\sum_{\ell=1}^{a_{k}}\int_{0}%
^{R}\left\vert w_{\left(  k,\ell\right)  }\left(  r\right)  \right\vert
rdr<\infty. \label{wNorm}%
\end{equation}
Hence, if the weight $w$ satisfies the summability condition, the cubature
coefficients satisfy the important stability inequality (\ref{wRemarkable}),
and experimental evidences show that \emph{all} coefficients are in general
very small. This inequality is proved in Theorem \ref{TChebyshev} by an
application of the famous Chebyshev extremal property for the Gauss-Jacobi
quadrature, see Theorem 4.1 in Chapter 4 of \cite{kreinNudelman}.

Let us give a short outline of the paper.

In Section \ref{SdiscretePoly} we provide basic properties of the discrete
polyharmonic cubature formulas: the summability condition for the weight
function implies that $I_{\left(  N,M,K\right)  }^{\text{poly}}$ are uniformly
bounded functionals (in the parameters $N,M,K)$ on the set of all polynomials.
Under this assumption it follows that for $N,M,K\longrightarrow\infty,$ the
value $I_{\left(  N,M,K\right)  }^{\text{poly}}\left(  f\right)  $ converges
to $I_{w}\left(  f\right)  $ for any function $f$ which is continuous on the
closed disk with radius $R.$ Moreover we show in Theorem \ref{TexactSpace}
that our cubature formula $I_{\left(  N,M,K\right)  }^{\text{poly}}$ is exact
for all polynomials of the type $f\left(  x\right)  =r^{2s+k}Y_{\left(
k,\ell\right)  }\left(  \varphi\right)  $ where $0\leq s\leq2N-1,$ $0\leq
k\leq M-1-K,$ and $\ell=1,2,...,a_{k},$ i.e.
\[
I_{\left(  N,M,K\right)  }^{\text{poly}}\left(  f\right)  =\int_{D_{R}%
}f\left(  x\right)  w\left(  x\right)  dx.
\]

Section \ref{SerrorPoly} is devoted to error estimates for the discrete
polyharmonic cubature. The error bounds are a sum of the error bounds of three
successive approximations: 1. the approximation of the weight function $w$ as
a Fourier series -- involving the parameter $K;$ 2. the approximation by the
one-dimensional quadrature formula in radial direction -- involving the
parameter $N;$ 3. the approximation by the Discrete Fourier Transform --
involving the parameter $M.$

In Section \ref{Sexperiments} we provide experimental results for the discrete
polyharmonic cubature with respect to the first weight function $w^{\left(
1\right)  }$ in (\ref{eqweight}) for four different types of test functions,
and compare the results with those obtained by the piece-wise midpoint rule
(\cite{DeBa01}) and the rule of Peirce (\cite{Peir57}), which are two methods
used widely in practice, in particular in the Meshless Petrov-Galerkin method,
see \cite{DeBa00}, \cite{DeBa01}, \cite{MaPi10}, \cite{MFPG07}. Our methods
have much higher accuracy than all other methods which might be explained by
the fact that the weight function $w^{\left(  1\right)  }$ has a discontinuity
in $0$ which has a strong negative influence on the results of the usual
cubature formulas.

Section \ref{Sexperiments2} contains experimental results for the second
weight function $w^{\left(  2\right)  }$ in (\ref{eqweight}). Since the weight
function $w^{\left(  2\right)  }$ has an infinite Fourier series it has been
expected that the discrete polyharmonic cubature now would perform weaker than
for the first weight function. However, even in this case the accuracy has
been extremely good provided that the parameter $K$ is large enough. Finally,
in Section \ref{Spractical} we comment on practical aspects for the numerical implementation.

\section{The discrete polyharmonic cubature \label{SdiscretePoly}}

At first we introduce the notion of pseudo-positive (or more general,
pseudo-definite) weight functions which plays a central role in the discrete
polyharmonic cubature formula.

\begin{definition}
Let $f\left(  x,y\right)  $ be a function on the disc with center $0$ and
radius $R$ which is continuous except for $0.$ Then $f$ is called
pseudo-definite if every Fourier coefficient $f_{\left(  k,\ell\right)
}\left(  r\right)  $ of $f$ has a definite sign, thus either $f_{\left(
k,\ell\right)  }\left(  r\right)  \geq0$ for all $r\in\left(  0,R\right)  $ or
$f_{\left(  k,\ell\right)  }\left(  r\right)  \leq0$ for all $r\in\left(
0,R\right)  .$ The function $f$ is called pseudo-positive if $f_{\left(
k,\ell\right)  }\left(  r\right)  \geq0$ for all $r\in\left(  0,R\right)  $
for all $k\in\mathbb{N}_{0},\ell=1,...,a_{k}$.
\end{definition}

A simple example of a pseudo-positive function is the Poisson kernel on the
unit disk given by
\[
P\left(  x,y\right)  =\frac{1-x^{2}-y^{2}}{\left(  x-1\right)  ^{2}+y^{2}%
}=\frac{1-r^{2}}{1-2r\cos\varphi+r^{2}}=1+\sum_{k=1}^{\infty}2r^{k}\cos
k\varphi.
\]

Let us recall that a \emph{cubature rule} is just an expression of the type
\[
C\left(  f\right)  =\sum_{j=1}^{N}c_{j}f\left(  x_{j}\right)
\]
where $x_{1},...,x_{N}$ are pairwise different points in $\mathbb{R}^{n}$,
called nodes, and $c_{j}\neq0$ are real constants called weights. An important
property of a cubature rule, which is also a basis for widely used method for
construction of cubature formulas, is the exactness for a subspace of
functions: a cubature formula $C$ is \emph{exact} on a subspace of functions
$U$ with respect to a weight function $w\left(  x\right)  $ if
\[
C\left(  f\right)  =\int f\left(  x\right)  w\left(  x\right)  dx\qquad\text{
for all }f\in U.
\]
If $U$ is the subspace $\mathcal{P}_{m}\mathbb{\ }$of all polynomials of
degree $\leq m$ then one says that a \emph{cubature formula }$C$\emph{\ has
degree }$m$ if $C$ is exact on $\mathcal{P}_{m}$ and if there exists a
polynomial $f$ of degree $m+1$ such that $C\left(  f\right)  \neq I_{w}\left(
f\right)  .$ Exactness of degree $m$ is also useful for providing error
estimates using Taylor series, an approach which was emphasized already by R.
von Mises, see \cite{Mise36}, \cite{Mise53}.

Recall that the \textbf{Discrete Polyharmonic Cubature formula } with
parameters $\left(  N,M,K\right)  $ is defined for any continuous function $f$
on the disk of radius $R$ by formula (\ref{DiscrPolyhCubature+}):
\begin{equation}
I_{\left(  N,M,K\right)  }^{\text{poly}}\left(  f\right)  :=\frac{\pi}{M}%
{\displaystyle\sum_{k=0}^{K}}
{\displaystyle\sum_{\ell=1}^{a_{k}}}
{\displaystyle\sum_{j=1}^{N}}
{\displaystyle\sum_{s=1}^{M}}
\lambda_{j,\left(  k,\ell\right)  }t_{j,\left(  k,\ell\right)  }^{-\frac{k}%
{2}}Y_{\left(  k,\ell\right)  }\left(  \frac{2\pi s}{M}\right)  \times
f\left(  \sqrt{t_{j,\left(  k,\ell\right)  }}e^{2\pi i\frac{s}{M}}\right)  .
\label{DiscrPolyhCubature2}%
\end{equation}
Thus $I_{\left(  N,M,K\right)  }^{\text{poly}}\left(  f\right)  $ is a
cubature formula where the weights are \emph{real} numbers. At first we
establish the following subtle estimate (analogous to an estimate proved for
the polyharmonic cubature formula in \cite{kounchevRenderArxiv}\footnote{The
notations in this reference differ from the present: Here we have put
$d\mu_{\left(  k,\ell\right)  }\left(  r\right)  =\int_{0}^{2\pi}Y_{\left(
k,\ell\right)  }\left(  \varphi\right)  d\mu\left(  re^{i\varphi}\right)  $
while in \cite{kounchevRenderArxiv} we used to work with .$d\mu_{\left(
k,\ell\right)  }\left(  r\right)  =r^{k}\int_{0}^{2\pi}Y_{\left(
k,\ell\right)  }\left(  \varphi\right)  d\mu\left(  re^{i\varphi}\right)  .$}):

\begin{theorem}
\label{TChebyshev} Let $w$ be a pseudo-definite weight function with Fourier
coefficients $w_{\left(  k,\ell\right)  }$, and let $f$ be bounded on the
closed disk of radius $R$ with supremum norm%
\[
\left\Vert f\right\Vert _{\infty}:=\sup_{x^{2}+y^{2}\leq R^{2}}\left\vert
f\left(  x,y\right)  \right\vert .
\]
Then
\[
\left\vert I_{\left(  N,M,K\right)  }^{\text{poly}}\left(  f\right)
\right\vert \leq\sqrt{\pi}\left\Vert f\right\Vert _{\infty}%
{\displaystyle\sum_{k=0}^{K}}
{\displaystyle\sum_{\ell=1}^{a_{k}}}
\int_{0}^{R}\left\vert w_{\left(  k,\ell\right)  }\left(  r\right)
\right\vert rdr\leq\sqrt{\pi}\left\Vert f\right\Vert _{\infty}\left\Vert
w\right\Vert .
\]
Also, the coefficients of formula (\ref{DiscrPolyhCubature2}) satisfy
inequality (\ref{wRemarkable}).
\end{theorem}

\begin{proof}
Since $\left\vert Y_{\left(  k,\ell\right)  }\left(  \varphi\right)
\right\vert \leq1/\sqrt{\pi}$ for all $\varphi\in\left[  0,2\pi\right]  $ the
following estimate follows directly from (\ref{DiscrPolyhCubature2}):
\[
\left\vert I_{\left(  N,M,K\right)  }^{\text{poly}}\left(  f\right)
\right\vert \leq\sqrt{\pi}\left\Vert f\right\Vert _{\infty}%
{\displaystyle\sum_{k=0}^{K}}
{\displaystyle\sum_{\ell=1}^{a_{k}}}
{\displaystyle\sum_{j=1}^{N}}
\left\vert \lambda_{j,\left(  k,\ell\right)  }\right\vert t_{j,\left(
k,\ell\right)  }^{-\frac{k}{2}}.
\]
Following (\ref{eqexactkl}) we define the function $G_{N}^{\left(
k,\ell\right)  }\left(  h\right)  :=%
{\displaystyle\sum_{j=1}^{N}}
\left\vert \lambda_{j,\left(  k,\ell\right)  }\right\vert h\left(
t_{j,\left(  k,\ell\right)  }\right)  ,$ which due to the pseudo-definiteness
of $w$, is the $N$-point Gauss-Jacobi quadrature for the integral
\[
I_{\left(  k,\ell\right)  }\left(  h\right)  :=\frac{1}{2}\int_{0}^{R^{2}%
}h\left(  \rho\right)  \cdot\rho^{k/2}\left\vert w_{\left(  k,\ell\right)
}\left(  \sqrt{\rho}\right)  \right\vert d\rho.
\]
Now we apply the Chebyshev extremal property of the Gau\ss --Jacobi quadrature
(see Theorem 4.1 in Chapter 4 of \cite{kreinNudelman}) which shows that
\[
G_{N}^{\left(  k,\ell\right)  }\left(  h\right)  \leq I_{\left(
k,\ell\right)  }\left(  h\right)
\]
holds for any $2N-$smooth function $h\left(  r\right)  $ with $h^{\left(
2N\right)  }\left(  \rho\right)  \geq0$ for all $\rho>0.$ Applying this to the
function $h\left(  \rho\right)  =\rho^{-k/2},$ gives the inequality
\begin{equation}
G_{N}^{\left(  k,\ell\right)  }\left(  \rho^{-k/2}\right)  =%
{\displaystyle\sum_{j=1}^{N}}
\left\vert \lambda_{j,\left(  k,\ell\right)  }\right\vert t_{j,\left(
k,\ell\right)  }^{-\frac{k}{2}}\leq\frac{1}{2}\int_{0}^{R^{2}}\left\vert
w_{\left(  k,\ell\right)  }\left(  \sqrt{\rho}\right)  \right\vert d\rho
=\int_{0}^{R}\left\vert w_{\left(  k,\ell\right)  }\left(  r\right)
\right\vert rdr. \label{GNkl}%
\end{equation}

\end{proof}

From Theorem \ref{TChebyshev} by using standard results from functional
analysis (see \cite[p. 351]{Davis}) we obtain the following important:

\begin{corollary}
Suppose that the Fourier coefficients of the pseudo-definite weight function
$w$ satisfy the summability condition (see (\ref{wNorm}))
\[
\left\Vert w\right\Vert <\infty.
\]
Then the discrete polyharmonic cubature $I_{\left(  N,M,K\right)
}^{\text{poly}}\left(  f\right)  $ converges to the integral $I_{w}\left(
f\right)  $ for each continuous function on the closed disc with radius $R$
when $N,M,K$ are approaching infinity.
\end{corollary}

Let us recall that the \emph{trapezoidal sum} (see \cite[p. 111]{DaRa75}),
also called the \emph{composite trapezoidal rule} in \cite[p. 155]{Gaut97}, is
defined on the set of all $2\pi$-periodic functions $g$ on the line, by
\begin{equation}
T_{M}\left(  g\right)  :=\frac{2\pi}{M}%
{\displaystyle\sum_{s=1}^{M}}
g\left(  \frac{2\pi s}{M}\right)  . \label{trapezoidal}%
\end{equation}
Let now $f\left(  x,y\right)  $ be defined on the disk with radius $R$ and
define $f_{r}\left(  \varphi\right)  :=f\left(  r\cos\varphi,r\sin
\varphi\right)  $ for $r\in\left[  0,R\right]  .$ Then
\[
T_{M}\left(  f_{r}\cdot Y_{\left(  k,\ell\right)  }\right)  =\frac{2\pi}{M}%
{\displaystyle\sum_{s=1}^{M}}
f\left(  r\cos\left(  \frac{2\pi s}{M}\right)  ,r\sin\left(  \frac{2\pi s}%
{M}\right)  \right)  Y_{\left(  k,\ell\right)  }\left(  \frac{2\pi ks}%
{M}\right)
\]
and the discrete polyharmonic cubature (\ref{DiscrPolyhCubature}) can be
written as
\begin{equation}
I_{\left(  N,M,K\right)  }^{\text{poly}}\left(  f\right)  =\frac{1}{2}%
{\displaystyle\sum_{k=0}^{K}}
{\displaystyle\sum_{\ell=1}^{a_{k}}}
{\displaystyle\sum_{j=1}^{N}}
\lambda_{j,\left(  k,\ell\right)  }\times t_{j,\left(  k,\ell\right)
}^{-\frac{k}{2}}T_{M}\left(  f_{\sqrt{t_{j,\left(  k,\ell\right)  }}%
}Y_{\left(  k,\ell\right)  }\right)  . \label{DiscrPolyhCubatureTrapez}%
\end{equation}

For the following result recall that $r^{k}Y_{\left(  k,\ell\right)  }\left(
\varphi\right)  $ is a polynomial of degree $k$ in the variables
$x=r\cos\varphi$ and $y=r\sin\varphi,$ cf. \cite{okbook}, \cite{sobolev}.

\begin{theorem}
\label{TexactSpace} Let $M$ and $K$ be natural numbers satisfying $M>K.$ Then
the discrete polyharmonic cubature $I_{\left(  N,M,K\right)  }^{\text{poly}}$
with parameters $\left(  N,M,K\right)  $ is exact on the linear subspace
generated by the polynomials
\[
r^{2s+k}Y_{\left(  k,\ell\right)  }\left(  \varphi\right)
\]
where $0\leq s\leq2N-1,$ $k\leq M-1-K,$ and $\ell=1,...,a_{k}.$
\end{theorem}

\begin{proof}
We shall use that $T_{M}\left(  g\right)  =\int_{0}^{2\pi}g\left(
\varphi\right)  d\varphi$ for all trigonometric polynomials $g$ of degree
$k\leq M-1$, see e.g. \cite{DaRa75}, p. $110$. Let $k\leq K$ be a natural
number and let $k_{1}\leq M-1-K$ be a natural number. Then it is easy to see
that the product $Y_{\left(  k_{1},\ell_{1}\right)  }\left(  \varphi\right)
Y_{\left(  k,\ell\right)  }\left(  \varphi\right)  $ is a trigonometric
polynomial of degree $k_{1}+k\leq M-1.$ It follows that
\begin{equation}
T_{M,0}\left(  Y_{\left(  k_{1},\ell_{1}\right)  }Y_{\left(  k,\ell\right)
}\right)  =\int_{0}^{2\pi}Y_{\left(  k_{1},\ell_{1}\right)  }\left(
\varphi\right)  Y_{\left(  k,\ell\right)  }\left(  \varphi\right)
d\varphi=\delta_{k,k_{1}}\delta_{\ell,\ell_{1}}, \label{TMYkl}%
\end{equation}
where $\delta_{i,j}$ is the Kronecker symbol. Now we consider the polynomial
\[
f^{\left(  s,k_{1},\ell_{1}\right)  }\left(  r\cos\varphi,r\sin\varphi\right)
=r^{2s}r^{k_{1}}Y_{\left(  k_{1},\ell_{1}\right)  }\left(  \varphi\right)
\]
with $s\leq2N-1$ and $k_{1}\leq M-1-K.$ Then by (\ref{TMYkl}) we obtain
\begin{align*}
I_{\left(  N,M,K\right)  }^{\text{poly}}\left(  f^{\left(  s,k_{1},\ell
_{1}\right)  }\right)   &  =\frac{1}{2}%
{\displaystyle\sum_{k=0}^{K}}
{\displaystyle\sum_{\ell=1}^{a_{k}}}
{\displaystyle\sum_{j=1}^{N}}
\lambda_{j,\left(  k,\ell\right)  }\times t_{j,\left(  k,\ell\right)
}^{-\frac{k}{2}}T_{M,0}\left(  t_{j,\left(  k,\ell\right)  }^{s}\left(
\sqrt{t_{j,\left(  k,\ell\right)  }}\right)  ^{k_{1}}Y_{\left(  k_{1},\ell
_{1}\right)  }Y_{\left(  k,\ell\right)  }\right) \\
&  =\frac{1}{2}%
{\displaystyle\sum_{k=0}^{K}}
{\displaystyle\sum_{\ell=1}^{a_{k}}}
{\displaystyle\sum_{j=1}^{N}}
\lambda_{j,\left(  k,\ell\right)  }t_{j,\left(  k,\ell\right)  }^{s}\int
_{0}^{2\pi}Y_{\left(  k_{1},\ell_{1}\right)  }\left(  \varphi\right)
Y_{\left(  k,\ell\right)  }\left(  \varphi\right)  d\varphi\\
&  =\frac{1}{2}%
{\displaystyle\sum_{j=1}^{N}}
\lambda_{j,\left(  k_{1},\ell_{1}\right)  }t_{j,\left(  k_{1},\ell_{1}\right)
}^{s}\\
&  =\int_{0}^{R}r^{2s+k_{1}}w_{\left(  k_{1},\ell_{1}\right)  }\left(
r\right)  rdr=\int_{0}^{R}f_{\left(  k_{1},\ell_{1}\right)  }^{\left(
s,k_{1},\ell_{1}\right)  }\left(  r\right)  w_{\left(  k_{1},\ell_{1}\right)
}\left(  r\right)  rdr
\end{align*}
In the last row we used formula (\ref{eqexactkl}), due to the fact that the
Gauss-Jacobi quadrature for each $\left(  k_{1},\ell_{1}\right)  $ is exact
for all polynomials of degree $2N-1.$ In the Fourier series expansion of the
polynomial $f^{\left(  s,k_{1},\ell_{1}\right)  }$ the only non-zero
coefficient is $f_{\left(  k_{1},\ell_{1}\right)  }^{\left(  s,k_{1},\ell
_{1}\right)  }\left(  r\right)  =r^{2s}r^{k_{1}}.$ Hence, by formulas
(\ref{Iwf2}), (\ref{eqIpoly}) we obtain
\[
I_{\left(  N,M,K\right)  }^{\text{poly}}\left(  f^{\left(  s,k_{1},\ell
_{1}\right)  }\right)  =I_{w}\left(  f^{\left(  s,k_{1},\ell_{1}\right)
}\right)  .
\]

\end{proof}

\begin{remark}
In \cite{Ahli62} one can find a nice account of the early history of numerical
multivariate integration until the 1950's, commencing with the work of J.C.
Maxwell \cite{Maxw1877} in 1877, of P. Appel \cite{Appe1890}, \cite{Appe26}
and H. Bourget \cite{Bour1898}. A. Ahlin \cite{Ahli62} stresses the fact that
for many ad hoc formulae there are no error estimates available, and he
provided error estimates for the case of product type measures. Since the
1950's the literature has grown considerably and very good surveys until the
1970's can be found in the books \cite{DaRa75}, \cite{Engels}, \cite{krylov},
\cite{sobolev2} and \cite{stroudBook}. For a recent survey on numerical
integration rules we refer to \cite{Cool03}. For numerical integration rules
especially for the disc we refer to \cite{CoHa87}, \cite{CoKi00},
\cite{Haeg76}, \cite{KiSo97},, \cite{VeCo92} and older work in \cite{AlCo58},
\cite{Albr60} , \cite{AlEn76}; see also \cite{vioreanuRokhlin},
\cite{xiaoGimbutas}.
\end{remark}

\section{Error Estimates for the discrete polyharmonic
cubature\label{SerrorPoly}}

In this section we derive error estimates for the discrete polyharmonic
cubature by considering the following chain of identities and approximations:
at first the integral (\ref{eqint}) is equal to (\ref{eqcentral2}), and a
simple change of variables leads to (\ref{r1}); then the infinite sum is
approximated by a finite sum (depending on $K),$ and next we employ the
Gauss-Jacobi quadratures (depending on $N),$ and finally we use the "Discrete
Fourier transform" (depending on $M)$ in order to obtain the discrete
polyharmonic cubature $I_{\left(  N,M,K\right)  }^{\text{poly}}\left(
f\right)  $:
\begin{align}
&  \int_{D_{R}}f\left(  x\right)  w\left(  x\right)  dx\nonumber\\
&  =\sum_{k,\ell}\int_{0}^{R}f_{\left(  k,\ell\right)  }\left(  r\right)
w_{\left(  k,\ell\right)  }\left(  r\right)  rdr=\frac{1}{2}\sum_{k=0}%
^{\infty}\sum_{\ell=1}^{a_{k}}\int_{0}^{R^{2}}f_{\left(  k,\ell\right)
}\left(  \sqrt{\rho}\right)  w_{\left(  k,\ell\right)  }\left(  \sqrt{\rho
}\right)  d\rho\label{r1}\\
&  \approx\frac{1}{2}\sum_{k=0}^{K}\sum_{\ell=1}^{a_{k}}\int_{0}^{R^{2}%
}f_{\left(  k,\ell\right)  }\left(  \sqrt{\rho}\right)  \rho^{-\frac{k}{2}%
}\times\rho^{\frac{k}{2}}w_{\left(  k,\ell\right)  }\left(  \sqrt{\rho
}\right)  d\rho\label{r2}\\
&  \approx\frac{1}{2}\sum_{k=0}^{K}\sum_{\ell=1}^{a_{k}}\sum_{j=1}%
^{N}f_{\left(  k,\ell\right)  }\left(  \sqrt{t_{j,\left(  k,\ell\right)  }%
}\right)  \times t_{j,\left(  k,\ell\right)  }^{-\frac{k}{2}}\lambda
_{j,\left(  k,\ell\right)  }=:I_{\left(  N,\infty,K\right)  }^{\text{poly}%
}\left(  f\right) \label{r3}\\
&  \approx\frac{1}{2}\sum_{k=0}^{K}\sum_{\ell=1}^{a_{k}}\sum_{j=1}%
^{N}f_{\left(  k,\ell\right)  }^{\left(  M\right)  }\left(  \sqrt{t_{j,\left(
k,\ell\right)  }}\right)  \times t_{j,\left(  k,\ell\right)  }^{-\frac{k}{2}%
}\lambda_{j,\left(  k,\ell\right)  }=I_{\left(  N,M,K\right)  }^{\text{poly}%
}\left(  f\right)  \label{r4}%
\end{align}

The error between formulas (\ref{r3}) and (\ref{r4}) can be estimated as follows:

\begin{theorem}
\label{TKapproximation} For $f\in C^{2D+1}\left(  B_{R}\right)  $ the
following estimate holds
\[
\left\vert I_{\left(  N,\infty,K\right)  }^{\text{poly}}\left(  f\right)
-I_{\left(  N,M,K\right)  }^{\text{poly}}\left(  f\right)  \right\vert
\leq\frac{2\pi\zeta\left(  2D+1\right)  }{M^{2D+1}}%
{\displaystyle\sum_{k=0}^{K}}
{\displaystyle\sum_{\ell=1}^{a_{k}}}
M_{k,\ell}\left(  f\right)  \int_{0}^{R}w_{\left(  k,\ell\right)  }\left(
r\right)  rdr
\]
where
\[
M_{k,\ell}\left(  f\right)  :=\sup_{r\in\left[  0,R\right]  }\sup_{\varphi
\in\left[  0,2\pi\right]  }\left\vert \frac{d^{2D+1}}{d\varphi^{2D+1}}\left[
f\left(  re^{i\varphi}\right)  Y_{\left(  k,\ell\right)  }\left(
\varphi\right)  \right]  \right\vert .
\]
Here $\zeta$ denotes the Riemann zeta function.
\end{theorem}

\begin{proof}
In \cite[p. 110]{DaRa75}, for periodic $g\in C^{2D+1}\left[  0,2\pi\right]  $
it is shown that
\begin{equation}
\left\vert \int_{0}^{2\pi}g\left(  \varphi\right)  d\varphi-T_{M}\left(
g\right)  \right\vert \leq\frac{C_{g,D}}{M^{2D+1}}, \label{Estgg}%
\end{equation}
where the constant $C_{g,D}$ is given by
\[
C_{g,D}:=4\pi\zeta\left(  2D+1\right)  \sup_{\varphi\in\left[  0,2\pi\right]
}\left\vert \frac{d^{2D+1}}{d\varphi^{2D+1}}\left[  g\left(  \varphi\right)
\right]  \right\vert .
\]
Formula (\ref{DiscrPolyhCubatureTrapez}) and the triangle inequality show that
for $f_{r}\left(  e^{i\varphi}\right)  =f\left(  re^{i\varphi}\right)  $
\[
\left\vert I_{\left(  N,\infty,K\right)  }^{\text{poly}}\left(  f\right)
-I_{\left(  N,M,K\right)  }^{\text{poly}}\left(  f\right)  \right\vert
\leq\frac{1}{2}%
{\displaystyle\sum_{k=0}^{K}}
{\displaystyle\sum_{\ell=1}^{a_{k}}}
{\displaystyle\sum_{j=1}^{N}}
\lambda_{j,\left(  k,\ell\right)  }t_{j,\left(  k,\ell\right)  }^{-\frac{k}%
{2}}\left\vert f_{\left(  k,\ell\right)  }\left(  \sqrt{t_{j,\left(
k,\ell\right)  }}\right)  -T_{M}\left(  f_{\sqrt{t_{j,\left(  k,\ell\right)
}}}Y_{\left(  k,\ell\right)  }\right)  \right\vert .
\]
Now we use inequality (\ref{Estgg}) for estimating the term in the middle
row:\
\begin{align*}
&  \left\vert f_{\left(  k,\ell\right)  }\left(  \sqrt{t_{j,\left(
k,\ell\right)  }}\right)  -T_{M}\left(  f_{\sqrt{t_{j,\left(  k,\ell\right)
}}}Y_{\left(  k,\ell\right)  }\right)  \right\vert \\
&  =\left\vert \int_{0}^{2\pi}f\left(  \sqrt{t_{j,\left(  k,\ell\right)  }%
}e^{i\varphi}\right)  Y_{\left(  k,\ell\right)  }\left(  \varphi\right)
d\varphi-T_{M}\left(  f_{\sqrt{t_{j,\left(  k,\ell\right)  }}}Y_{\left(
k,\ell\right)  }\right)  \right\vert \\
&  \leq\frac{4\pi\zeta\left(  2D+1\right)  M_{k,\ell}\left(  f\right)
}{M^{2D+1}}.
\end{align*}
Hence, by inequality (\ref{GNkl}), we obtain the final result:\
\begin{align*}
\left\vert I_{\left(  N,\infty,K\right)  }^{\text{poly}}\left(  f\right)
-I_{\left(  N,M,K\right)  }^{\text{poly}}\left(  f\right)  \right\vert  &
\leq\frac{2\pi\zeta\left(  2D+1\right)  }{M^{2D+1}}%
{\displaystyle\sum_{k=0}^{K}}
{\displaystyle\sum_{\ell=1}^{a_{k}}}
M_{k,\ell}\left(  f\right)
{\displaystyle\sum_{j=1}^{N}}
\lambda_{j,\left(  k,\ell\right)  }t_{j,\left(  k,\ell\right)  }^{-\frac{k}%
{2}}.\\
&  \leq\frac{2\pi\zeta\left(  2D+1\right)  }{M^{2D+1}}%
{\displaystyle\sum_{k=0}^{K}}
{\displaystyle\sum_{\ell=1}^{a_{k}}}
M_{k,\ell}\left(  f\right)  \int_{0}^{R}w_{\left(  k,\ell\right)  }\left(
r\right)  rdr.
\end{align*}

\end{proof}

The error bound between (\ref{r1}) and (\ref{r2}) is considered in the following:

\begin{theorem}
\label{TtrapezoidalApprox} Let $f\in C^{2p}\left(  \overline{D_{R}}\right)  $
for some integer $p\geq1.$ Assume that the weight function $w$ satisfies
\[
\frac{1}{2}\sum_{k=0}^{\infty}\sum_{\ell=1}^{a_{k}}\int_{0}^{R^{2}}\left\vert
w_{\left(  k,\ell\right)  }\left(  \sqrt{\rho}\right)  \right\vert d\rho
=\sum_{k=0}^{\infty}\sum_{\ell=1}^{a_{k}}\int_{0}^{R^{2}}\left\vert w_{\left(
k,\ell\right)  }\left(  r\right)  \right\vert rdr=\left\Vert w\right\Vert
<\infty
\]
Then
\[
\left\vert \sum_{k=K+1}^{\infty}\sum_{\ell=1}^{a_{k}}\int_{0}^{R}f_{\left(
k,\ell\right)  }\left(  \sqrt{\rho}\right)  \rho^{-\frac{k}{2}}\times
\rho^{\frac{k}{2}}w_{\left(  k,\ell\right)  }\left(  \sqrt{\rho}\right)
d\rho\right\vert \leq\frac{\sqrt{2\pi}}{K^{2p}}\left\Vert w\right\Vert
\left\Vert \frac{\partial^{2}}{\partial\varphi^{2}}f\left(  re^{i\varphi
}\right)  \right\Vert _{\infty}.
\]

\end{theorem}

%

\proof
By applying a standard techniques (see e.g. Theorem $10.19$ in \cite{okbook}),
we obtain
\begin{align*}
f_{\left(  k,\ell\right)  }\left(  r\right)   &  =\int_{0}^{2\pi}f\left(
re^{i\varphi}\right)  Y_{\left(  k,\ell\right)  }\left(  \varphi\right)
d\varphi=\frac{1}{k^{2p}}\int_{0}^{2\pi}f\left(  re^{i\varphi}\right)
\frac{\partial^{2p}}{\partial\varphi^{2p}}Y_{\left(  k,\ell\right)  }\left(
\varphi\right)  d\varphi\\
&  =\frac{1}{k^{2p}}\int_{0}^{2\pi}\frac{\partial^{2p}}{\partial\varphi^{2p}%
}f\left(  re^{i\varphi}\right)  Y_{\left(  k,\ell\right)  }\left(
\varphi\right)  d\varphi.
\end{align*}
We can estimate this inequality by using the Cauchy-Schwarz inequality and the
orthonormality of $Y_{\left(  k,\ell\right)  }\left(  \varphi\right)  $,
arriving at
\[
\left\vert f_{\left(  k,\ell\right)  }\left(  r\right)  \right\vert \leq
\frac{\sqrt{2\pi}}{k^{2p}}\left\Vert \frac{\partial^{2}}{\partial\varphi^{2}%
}f\left(  re^{i\varphi}\right)  \right\Vert _{\infty}.
\]
Then
\begin{align*}
&  \left\vert \sum_{k=K+1}^{\infty}\sum_{\ell=1}^{a_{k}}\int_{0}^{R}f_{\left(
k,\ell\right)  }\left(  \sqrt{\rho}\right)  w_{\left(  k,\ell\right)  }\left(
\sqrt{\rho}\right)  d\rho\right\vert \\
&  \leq\sum_{k=K+1}^{\infty}\sum_{\ell=1}^{a_{k}}\int_{0}^{R^{2}}\left\vert
f_{\left(  k,\ell\right)  }\left(  \sqrt{\rho}\right)  \right\vert \left\vert
w_{\left(  k,\ell\right)  }\left(  \sqrt{\rho}\right)  \right\vert d\rho\\
&  \leq\sqrt{2\pi}\left\Vert \frac{\partial^{2}}{\partial\varphi^{2}}f\left(
re^{i\varphi}\right)  \right\Vert _{\infty}\times\sum_{k=K+1}^{\infty}%
\sum_{\ell=1}^{a_{k}}\frac{1}{k^{2p}}\int_{0}^{R^{2}}\left\vert w_{\left(
k,\ell\right)  }\left(  \sqrt{\rho}\right)  \right\vert d\rho\\
&  \leq\sqrt{2\pi}\left\Vert \frac{\partial^{2}}{\partial\varphi^{2}}f\left(
re^{i\varphi}\right)  \right\Vert _{\infty}\times\frac{1}{K^{2p}}\sum
_{k=K+1}^{\infty}\sum_{\ell=1}^{a_{k}}\int_{0}^{R^{2}}\left\vert w_{\left(
k,\ell\right)  }\left(  \sqrt{\rho}\right)  \right\vert d\rho\\
&  \leq\sqrt{2\pi}\left\Vert \frac{\partial^{2}}{\partial\varphi^{2}}f\left(
re^{i\varphi}\right)  \right\Vert _{\infty}\times\frac{1}{K^{2p}}\left\Vert
w\right\Vert .
\end{align*}
This ends the proof.%

\endproof

The error bound between (\ref{r2}) and (\ref{r3}) is a Markov type estimate
(see Theorem $14.2.2$ in \cite{Davis}, and Theorem $44$ in
\cite{kounchevRenderArxiv}).\ Using the notations (\ref{nodes}),
(\ref{coefs}), (\ref{eqexactkl}) we prove:

\begin{theorem}
\label{TGaussApproximation} Let $\kappa_{N,\left(  k,\ell\right)  }$ be the
leading coefficient of the $N$-th degree orthonormal polynomial $Q_{N,\left(
k,\ell\right)  }$ with respect to the measure $\rho^{\frac{k}{2}}w_{\left(
k,\ell\right)  }\left(  \sqrt{\rho}\right)  $ on the interval $\left[
0,R^{2}\right]  $. Then for every function $f\in C^{2N}\left(  \overline
{D_{R}}\right)  $ and every index $\left(  k,\ell\right)  $ we have
\begin{align*}
I_{\left(  k,\ell\right)  }  &  :=\left\vert \int_{0}^{R^{2}}f_{\left(
k,\ell\right)  }\left(  \sqrt{\rho}\right)  w_{\left(  k,\ell\right)  }\left(
\sqrt{\rho}\right)  d\rho-\sum_{j=1}^{N}f_{\left(  k,\ell\right)  }\left(
\sqrt{t_{j,\left(  k,\ell\right)  }}\right)  \times t_{j,\left(
k,\ell\right)  }^{-\frac{k}{2}}\lambda_{j,\left(  k,\ell\right)  }\right\vert
\\
&  \leq\frac{1}{\left(  2N\right)  !\kappa_{N,\left(  k,\ell\right)  }^{2}%
}\left\Vert \frac{d^{2N}}{d\rho^{2N}}g_{\left(  k,\ell\right)  }\left(
\rho\right)  \right\Vert _{\infty},
\end{align*}
where $g_{\left(  k,\ell\right)  }\left(  \rho\right)  :=\rho^{-\frac{k}{2}%
}f_{\left(  k,\ell\right)  }\left(  \sqrt{\rho}\right)  $. The difference
between (\ref{r2}) and (\ref{r3}) is bounded by
\[
\sum_{k=0}^{K}\sum_{\ell=1}^{a_{k}}I_{\left(  k,\ell\right)  }\leq\frac
{1}{\left(  2N\right)  !}\sum_{k=0}^{K}\sum_{\ell=1}^{a_{k}}\frac{1}%
{\kappa_{N,\left(  k,\ell\right)  }^{2}}\left\Vert \frac{d^{2N}}{d\rho^{2N}%
}g_{\left(  k,\ell\right)  }\left(  \rho\right)  \right\Vert _{\infty}.
\]

\end{theorem}

%

\proof
It is easy to see that $g_{\left(  k,\ell\right)  }\left(  \rho\right)
=\rho^{-\frac{k}{2}}f_{\left(  k,\ell\right)  }\left(  \sqrt{\rho}\right)  $
is $2N$-times differentiable in the interval $\left(  0,R^{2}\right)  .$
Hence, to the function $g_{\left(  k,\ell\right)  }\left(  \rho\right)
=f_{\left(  k,\ell\right)  }\left(  \sqrt{\rho}\right)  \rho^{-\frac{k}{2}}\in
C^{2N\ }\left(  0,R^{2}\right)  ,$ we may apply Markov's Theorem $14.2.2$ in
\cite{Davis}, and we obtain
\[
I_{\left(  k,\ell\right)  }\leq\frac{1}{\left(  2N\right)  !\kappa_{N,\left(
k,\ell\right)  }^{2}}\left\Vert \frac{d^{2N}}{d\rho^{2N}}g_{\left(
k,\ell\right)  }\left(  \rho\right)  \right\Vert _{\infty}.
\]
This implies the estimate for the error between (\ref{r2}) and (\ref{r3})
\begin{align*}
&  \left\vert \sum_{k=0}^{K}\sum_{\ell=1}^{a_{k}}\int_{0}^{R^{2}}f_{\left(
k,\ell\right)  }\left(  \rho\right)  w_{\left(  k,\ell\right)  }\left(
\sqrt{\rho}\right)  d\rho-\sum_{k=0}^{K}\sum_{\ell=1}^{a_{k}}\sum_{j=1}%
^{N}f_{\left(  k,\ell\right)  }\left(  \sqrt{t_{j,\left(  k,\ell\right)  }%
}\right)  \times t_{j,\left(  k,\ell\right)  }^{-\frac{k}{2}}\lambda
_{j,\left(  k,\ell\right)  }\right\vert \\
&  \leq\sum_{k=0}^{K}\sum_{\ell=1}^{a_{k}}\left\vert \int_{0}^{R^{2}%
}g_{\left(  k,\ell\right)  }\left(  \rho\right)  \times\rho^{\frac{k}{2}%
}w_{\left(  k,\ell\right)  }\left(  \sqrt{\rho}\right)  d\rho-\sum_{j=1}%
^{N}f_{\left(  k,\ell\right)  }\left(  \sqrt{t_{j,\left(  k,\ell\right)  }%
}\right)  \times t_{j,\left(  k,\ell\right)  }^{-\frac{k}{2}}\lambda
_{j,\left(  k,\ell\right)  }\right\vert \\
&  \leq\frac{1}{\left(  2N\right)  !}\sum_{k=0}^{K}\sum_{\ell=1}^{a_{k}}%
\frac{1}{\kappa_{N,\left(  k,\ell\right)  }^{2}}\left\Vert \frac{d^{2N}}%
{d\rho^{2N}}g_{\left(  k,\ell\right)  }\left(  \rho\right)  \right\Vert
_{\infty}.
\end{align*}
%

\endproof

\begin{remark}
In the case when $\rho^{\frac{k}{2}}w_{\left(  k,\ell\right)  }\left(
\sqrt{\rho}\right)  $ are Jacobi weight functions explicit expressions for
$\kappa_{N,\left(  k,\ell\right)  }$ are well known, see e.g. \cite{Davis} or
Theorem $44$ in \cite{kounchevRenderArxiv}.
\end{remark}

\begin{remark}
If $g_{\left(  k,\ell\right)  }\left(  x\right)  $ is real on the real axis
and $2\pi$ periodic and if $g_{\left(  k,\ell\right)  }\left(  x+iy\right)  $
is holomorphic in a strip $\left\vert y\right\vert <\sigma,$ then one may use
estimates in \cite[p. 110]{DaRa75}, and apply them to obtain an alternative
estimate in Theorem \ref{TGaussApproximation}.
\end{remark}

\section{Experimental results for the weight function $w^{\left(  1\right)
}\left(  x,y\right)  $ \label{Sexperiments}}

In this section we want to test cubature formulas for integrals of the type
\[
I_{w}\left(  f\right)  =\int_{0}^{2\pi}\int_{0}^{R}f\left(  r\cos\varphi
,r\sin\varphi\right)  \cdot w^{\left(  1\right)  }\left(  r\cos\varphi
,r\sin\varphi\right)  \cdot rdrd\varphi
\]
for the weight function
\[
w^{\left(  1\right)  }\left(  x,y\right)  =\frac{1+x}{\sqrt{x^{2}+y^{2}}%
}\text{ }=\frac{1}{r}+\cos\varphi\ =\frac{\sqrt{2\pi}}{r}Y_{\left(
0,1\right)  }\left(  \varphi\right)  +\sqrt{\pi}Y_{\left(  1,1\right)
}\left(  \varphi\right)  .
\]
Since $w^{\left(  1\right)  }$ has only two Fourier coefficients we take $K=1$
in the discrete polyharmonic cubature (\ref{DiscrPolyhCubature}) with
parameters $(N,M,K).$ For given $N$ one has to construct the $N$-point
Gauss-Jacobi quadratures for the integrals
\[
\frac{1}{2}\int_{0}^{1}P\left(  \rho\right)  \rho^{-1/2}d\rho\quad\text{ for
}k=0,\text{ and }\quad\frac{1}{2}\int_{0}^{1}P\left(  \rho\right)  \rho
^{1/2}d\rho\quad\text{ for }k=1
\]
which are Jacobi-weight functions. Fortunately, there are excellent programmes
for determining the nodes and weights of the Gauss-Jacobi quadrature providing
high accuracy, see \cite{GautschiBOOKorthogonal}.

For our experiments in this section we consider four test functions:
\begin{align}
f_{0}\left(  x,y\right)   &  =1+x^{4}+y^{3},\label{fun1}\\
f_{1}\left(  x,y\right)   &  =1+\frac{x^{3}}{\sqrt{x^{2}+y^{2}}}+\frac{y^{7}%
}{x^{2}+y^{2}}=1+r^{2}\cos^{3}\varphi+r^{5}\sin^{7}\varphi,\label{fun2}\\
f_{2}\left(  x,y\right)   &  =\cos\left(  10x+20y\right)  ,\label{fun3}\\
f_{3}\left(  x,y\right)   &  =\left(  x^{2}+y^{2}\right)  ^{5/4}=r^{5/2}
\label{fun4b}%
\end{align}
The first test function $f_{0}$ is a polynomial of degree $4$, and the second
$f_{1}$ is not smooth at $0$. The function $f_{2}$ is of oscillatory type and
an example of a test function used in the package of Genz, see \cite{genz},
\cite{schurer} of the form (in our case $u=0)$
\[
\cos\left(  2\pi u+ax+by\right)  =\cos\left(  2\pi u\right)  \cos\left(
ax+by\right)  -\sin\left(  2\pi u\right)  \sin\left(  ax+by\right)  .
\]
At first we present the experiments for the Discrete Polyharmonic Cubature,
and then we compare it with two other standard rules which are applied to the
functions%
\[
g_{j}\left(  re^{i\varphi}\right)  :=f_{j}\left(  re^{i\varphi}\right)
w^{\left(  1\right)  }\left(  re^{i\varphi}\right)
\]
Note that $g_{j}$ is not continuous at $0$ for $j=0,1,2,$ and one might argue
that the standard rules do not perform too well for functions with a
discontinuity. For this reason we have included the function $f_{3}\left(
x,y\right)  $ for which $g_{3}\left(  r^{i\varphi}\right)  =r^{3/2}\left(
1+r\cos\varphi\right)  $ is clearly continuously differentiable, and our
discrete polyharmonic cubature formula (\ref{DiscrPolyhCubature+}) performs
better than the usual methods as well.

Let us note that the discrete polyharmonic cubature formula
(\ref{DiscrPolyhCubature+}) needs (at most) $\left(  2K-1\right)  \cdot N\cdot
M$ evaluation points. Since $w^{\left(  1\right)  }$ has only two non-zero
Fourier coefficients we need in this case at most $2N\cdot M$ evaluations points.

\subsection{Results for the Discrete Polyharmonic Cubature Formula}

In the following tables we present experimental results for the discrete
polyharmonic cubature where the number $M$ of points on the circles is chosen
to be equal to $9,25,63,83.$ The reason for this choice is that we used the
Fast Fourier transform for the implementation of the trapezoidal rule
(depending on $M)$. The number $N$ of concentric circles is chosen to be equal
to $10,15,25,35,50.$

Due to the exactness of the discrete polyharmonic cubature, the value
$I_{\left(  N,M,1\right)  }^{\text{poly}}\left(  f_{0}\right)  $, according to
Theorem \ref{TexactSpace}, must be identical with the true value of the
integral if $M$ satisfies $M\geq6,$ and $N\geq2.$ This is numerically
confirmed by our experiments: for all $N$ and $M$ as above specified we
obtained up to double precision that
\[
I_{1}\left(  f_{0}w^{\left(  1\right)  }\right)  =\frac{43}{20}\pi
\approx6.754424205218060.
\]
The second test function can also be integrated explicitly and the true value
is
\[
I_{1}\left(  f_{1}w^{\left(  1\right)  }\right)  =\frac{35}{16}\pi
\approx6.872\,233\,929727\,67.
\]
Our experimental results are contained in the following table:%

\begin{gather*}
\text{\textsc{Table} 1}\\%
\begin{tabular}
[c]{|l|l|l|l|l|}\hline
N/M & \textbf{9} & \textbf{25} & \textbf{63} & \textbf{83}\\\hline
\textbf{10} & 6.87224296287783 & 6.87224296287783 & 6.87224296287783 &
6.87224296287783\\\hline
\textbf{15} & 6.87223588060173 & 6.87223588060173 & 6.87223588060173 &
6.87223588060173\\\hline
\textbf{25} & 6.87223420205342 & 6.87223420205342 & 6.87223420205342 &
6.87223420205342\\\hline
\textbf{35} & 6.87223400297000 & 6.87223400297000 & 6.87223400297000 &
6.87223400297000\\\hline
\textbf{50} & 6.87223394775545 & 6.87223394775545 & 6.87223394775545 &
6.87223394775545\\\hline
Error &  &  &  & \\\hline
N/M & \textbf{9} & \textbf{25} & \textbf{63} & \textbf{83}\\\hline
\textbf{10} & 0.00000903315016 & 0.00000903315016 & 0.00000903315016 &
0.00000903315016\\\hline
\textbf{15} & 0.00000195087406 & 0.00000195087406 & 0.00000195087406 &
0.00000195087406\\\hline
\textbf{25} & 0.00000027232575 & 0.00000027232575 & 0.00000027232575 &
0.00000027232575\\\hline
\textbf{35} & 0.00000007324233 & 0.00000007324233 & 0.00000007324233 &
0.00000007324233\\\hline
\textbf{50} & 0.00000001802778 & 0.00000001802778 & 0.00000001802778 &
0.00000001802778\\\hline
\end{tabular}
\\
\text{\textbf{Discrete polyharmonic cubature} for }f_{1}w^{\left(  1\right)
}=\left(  1+r^{2}\cos^{3}\varphi+r^{5}\sin^{7}\varphi\right)  \left(  \frac
{1}{r}+\cos\varphi\right)
\end{gather*}

Obviously we obtain very good approximations: even for the case of
$180=2\times9\times10$ evaluations points the error is smaller than $10^{-5}.$
Note further that in the rows we do not obtain improvements when $M$ is
getting larger. This is due to the fact that $\varphi\longmapsto f_{1}\left(
re^{i\varphi}\right)  $ is a trigonometric polynomial of degree $\leq7,$ and
the exactness of the trapezoidal rule (\ref{trapezoidal}) implies that there
is no change when $M$ is larger than $7.$

Now we want to test the function (\ref{fun3})
\[
f_{2}\left(  x,y\right)  =\cos\left(  ax+by\right)  .
\]
It is not difficult to see that the first Fourier coefficient of $f_{2}$
satisfies
\begin{align*}
f_{2,\left(  0,1\right)  }\left(  r\right)   &  =\sqrt{2\pi}J_{0}\left(
\sqrt{a^{2}+b^{2}}r\right)  ,\qquad\text{and }\\
f_{2,\left(  1,1\right)  }\left(  r\right)   &  =f_{2,\left(  1,2\right)
}\left(  r\right)  =0,
\end{align*}
where $J_{n}$ is the Bessel function of the first kind of order $n$ defined
by
\begin{equation}
J_{n}\left(  x\right)  =\left(  \frac{x}{2}\right)  ^{n}\sum_{s=0}^{\infty
}\frac{\left(  -1\right)  ^{s}}{s!\left(  n+s\right)  !}\ \left(  \frac{x}%
{2}\right)  ^{2s}; \label{bessel}%
\end{equation}
cf. \cite{AAR99}. Then the first Fourier coefficient $g_{\left(  0,1\right)
}\left(  r\right)  $ of the function $g:=f_{2}\cdot w^{\left(  1\right)  }$
can be computed (using the orthogonality relations of spherical harmonics)
\[
g_{\left(  0,1\right)  }\left(  r\right)  =\frac{1}{\sqrt{2\pi}}f_{2,\left(
0,1\right)  }\left(  r\right)  w_{\left(  0,1\right)  }^{\left(  1\right)
}=J_{0}\left(  \sqrt{a^{2}+b^{2}}r\right)  \frac{\sqrt{2\pi}}{r}.
\]
It follows that
\[
I_{1}\left(  f_{2}w^{\left(  1\right)  }\right)  =\sqrt{2\pi}\int_{0}%
^{R}g_{\left(  0,1\right)  }\left(  r\right)  rdr=2\pi\int_{0}^{R}J_{0}\left(
\sqrt{a^{2}+b^{2}}r\right)  dr.
\]
Using the power series expansion of the Bessel function in (\ref{bessel}) the
integral can be evaluated up to any accuracy. For $a=10$ and $b=20$ we see
that up to double precision we have
\[
I_{w^{\left(  1\right)  }}\left(  f_{2}\right)  =I_{1}\left(  f_{2}w^{\left(
1\right)  }\right)  =0.301\,310\,995335\,215\,
\]
Our experiments for the discrete polyharmonic cubature provide the following table:%

\begin{gather*}
\text{\textsc{Table} 2}\\%
\begin{tabular}
[c]{|l|l|l|l|l|}\hline
N%
$\backslash$%
M & \textbf{9} & \textbf{25} & \textbf{63} & \textbf{83}\\\hline
\textbf{10} & -0.08102057453745 & 0.31409913156633 & 0.30131093100867 &
0.30131093100867\\\hline
\textbf{15} & -0.08102397430499 & 0.31409919589293 & 0.30131099533522 &
0.30131099533522\\\hline
\textbf{25} & -0.08102401217317 & 0.31409919589293 & 0.30131099533522 &
0.30131099533522\\\hline
\textbf{35} & -0.08102401237119 & 0.31409919589293 & 0.30131099533522 &
0.30131099533522\\\hline
\textbf{50} & -0.08102401237809 & 0.31409919589293 & 0.30131099533522 &
0.30131099533522\\\hline
Error &  &  &  & \\\hline
N%
$\backslash$%
M & \textbf{9} & \textbf{25} & \textbf{63} & \textbf{83}\\\hline
\textbf{10} & 0.38233156987266 & 0.01278813623111 & 0.00000006432655 &
0.00000006432655\\\hline
\textbf{15} & 0.38233496964021 & 0.01278820055772 & 0.00000000000000 &
0.00000000000000\\\hline
\textbf{25} & 0.38233500750838 & 0.01278820055772 & 0.00000000000000 &
0.00000000000000\\\hline
\textbf{35} & 0.38233500770641 & 0.01278820055772 & 0.00000000000000 &
0.00000000000000\\\hline
\textbf{50} & 0.38233500771330 & 0.01278820055772 & 0.00000000000000 &
0.00000000000000\\\hline
\end{tabular}
\\
\text{\textbf{Discrete polyharmonic cubature} for }f_{2}w^{\left(  1\right)
}=\left(  \cos\left(  10x+20y\right)  \right)  \left(  \frac{1}{r}+\cos
\varphi\right)  .
\end{gather*}

Note that the formula is sensitive with respect to the values of $M:$ if $M$
is $9$ then large deviations occur (even if $N$ is large) , for $M=25$ the
approximation error is $0.01$ and the number $N$ of circles does not influence
much the results. For $M=63$ and $N=10$ the approximation error is very
small:\
\[
0.301\,310\,995\,\allowbreak
335\,215-0.30131093100867=0.000\,000\,064326\,545\,
\]

Next we consider the test function $f_{3}\left(  x,y\right)  =r^{5/2},$ for
which the integrand $f_{3}w^{\left(  1\right)  }=r^{3/2}\left(  1+r\cos
\varphi\right)  $ is smooth. Then the explicit computation gives
\[
\int_{0}^{2\pi}\int_{0}^{1}r^{3/2}\left(  1+r\cos\varphi\right)
rdrd\varphi=1.795\,195\,802\,\allowbreak051\,31
\]
with the following table:%
\begin{gather*}
\text{\textsc{Table} 3}\\%
\begin{tabular}
[c]{|l|l|l|l|l|}\hline
N/M & \textbf{9} & \textbf{25} & \textbf{63} & \textbf{83}\\\hline
\textbf{10} & 1.79513323182095 & 1.79513323182095 & 1.79513323182095 &
1.79513323182095\\\hline
\textbf{15} & 1.79518029482336 & 1.79518029482336 & 1.79518029482336 &
1.79518029482336\\\hline
\textbf{25} & 1.79519315318245 & 1.79519315318245 & 1.79519315318245 &
1.79519315318245\\\hline
\textbf{35} & 1.79519497859942 & 1.79519497859942 & 1.79519497859942 &
1.79519497859942\\\hline
\textbf{50} & 1.79519556405565 & 1.79519556405565 & 1.79519556405565 &
1.79519556405565\\\hline
Error &  &  &  & \\\hline
N/M & \textbf{9} & \textbf{25} & \textbf{63} & \textbf{83}\\\hline
\textbf{10} & 0.000062570230356 & 0.000062570230356 & 0.000062570230356 &
0.000062570230356\\\hline
\textbf{15} & 0.000015507227945 & 0.000015507227945 & 0.000015507227945 &
0.000015507227945\\\hline
\textbf{25} & 0.000002648868861 & 0.000002648868861 & 0.000002648868861 &
0.000002648868861\\\hline
\textbf{35} & 0.000000823451885 & 0.000000823451885 & 0.000000823451885 &
0.000000823451885\\\hline
\textbf{50} & 0.000000237995663 & 0.000000237995663 & 0.000000237995663 &
0.000000237995663\\\hline
\end{tabular}
\\
\text{\textbf{Discrete polyharmonic cubature} for }f_{3}w^{\left(  1\right)
}=r^{5/2}\left(  \frac{1}{r}+\cos\varphi\right)
\end{gather*}

\subsection{Comparison with the piece-wise midpoint rule}

The \emph{piecewise midpoint quadrature rule }(see e..g. \cite{DeBa01}) is
rather geometric: subdivide the disk of radius $R$ by concentric circles of
radius
\[
r_{j}=\frac{j^{2}-j+1/3}{j-\frac{1}{2}}\frac{R}{N}\text{ }\approx\frac{j}%
{N}R\text{ for }j=1,...,N
\]
and radial half-lines with angle $2\pi i/M$ for $i=1,...,M.$ Then the integral
over each subdomain is approximated by the evaluation of the integrand at the
centroid of the sector multiplied by the area of the sector, given by $\left(
j-\frac{1}{2}\right)  \left(  \frac{R}{N}\right)  ^{2}\frac{2\pi}{M}.$
Formally, we define:

\begin{definition}
The \emph{piecewise midpoint quadrature rule, }is given by:
\begin{equation}
I_{N,M}^{\text{mid }}\left(  f\right)  :=\frac{2\pi R^{2}}{M\cdot N^{2}}%
{\displaystyle\sum_{j=1}^{N}}
{\displaystyle\sum_{s=1}^{M}}
\left(  j-\frac{1}{2}\right)  f\left(  r_{j}\cos\frac{2\pi\left(  s-\frac
{1}{2}\right)  }{M},r_{j}\sin\frac{2\pi\left(  s-\frac{1}{2}\right)  }%
{M}\right)  . \label{Rule1}%
\end{equation}

\end{definition}

The results for the first test function $f_{0}$ are contained in the following
table.
\begin{gather*}
\text{\textsc{Table} 4}\\%
\begin{tabular}
[c]{|l|l|l|}\hline
N=M & $I_{N,N}^{\text{mid }}\left(  f_{0}w^{\left(  1\right)  }\right)  $ &
Error\\\hline
\textbf{5} & 6.\thinspace\allowbreak293\thinspace948\thinspace149\thinspace
525\thinspace97 & 0.460\thinspace476\thinspace055\thinspace\allowbreak
692\thinspace09\\\hline
\textbf{10} & 6.\thinspace\allowbreak552\thinspace664\thinspace285\thinspace
742\thinspace99 & 0.201\thinspace759\thinspace919\thinspace\allowbreak
475\thinspace07\\\hline
\textbf{20} & 6.\thinspace\allowbreak652\thinspace725\thinspace619\thinspace
004\thinspace72 & 0.101\thinspace698\thinspace586\thinspace\allowbreak
213\thinspace34\\\hline
\textbf{100} & 6.\thinspace\allowbreak733\thinspace953\thinspace
574\thinspace717\thinspace90 & 0.020\thinspace470\thinspace630\thinspace
\allowbreak500\thinspace16\\\hline
\textbf{200} & 6.\thinspace\allowbreak744\thinspace180\thinspace
708\thinspace690\thinspace65 & 0.010\thinspace243\thinspace496\thinspace
\allowbreak527\thinspace41\\\hline
\end{tabular}
\\
\text{\textbf{Midpoint cubature for} }f_{0}w^{\left(  1\right)  }=\left(
1+x^{4}+y^{3}\right)  \left(  \frac{1}{r}+\cos\varphi\right) \\
\text{\textbf{True Value }is:\ }\frac{43}{20}\pi\approx
6.754\,424\,205\,218\,060
\end{gather*}
Note that the case $N=200,$ i.e. $40\,000$ evaluation points, still leads to
an error of $0.01.$ The discrete polyharmonic cubature was exact for this case.

For the second test function $f_{1}$ we have a similar pattern:%

\begin{gather*}
\text{\textsc{Table} 5}\\%
\begin{tabular}
[c]{|l|l|l|}\hline
N=M & $I_{N,N}^{\text{mid }}\left(  f_{1}w^{\left(  1\right)  }\right)  $ &
Error\\\hline
\textbf{5} & 6.479\thinspace185\thinspace720\thinspace\allowbreak
369\thinspace13 & 0.393048209358541\\\hline
\textbf{10} & 6.671\thinspace455\thinspace800\thinspace\allowbreak
859\thinspace80 & 0.200778128867871\\\hline
\textbf{20} & 6.770\thinspace780\thinspace856\thinspace\allowbreak
013\thinspace81 & 0.101453073713858\\\hline
\textbf{100} & 6.851\thinspace773\thinspace117\thinspace\allowbreak
091\thinspace46 & 0.020460812636194\\\hline
\textbf{200} & 6.861\thinspace992\thinspace887\thinspace\allowbreak
600\thinspace82 & 0.010241042126847\\\hline
\end{tabular}
\\
\text{\textbf{Midpoint cubature for} }f_{1}w^{\left(  1\right)  }=\left(
1+r^{2}\cos^{3}\varphi+r^{5}\sin^{7}\varphi\right)  \left(  \frac{1}{r}%
+\cos\varphi\right)  \text{ }\\
\text{\textbf{True value is} }\approx6.872\,233\,929\,727\,67
\end{gather*}
and we see that the error is quite big; for $200\times200$ evaluation points
it is bigger than $0.01.$

For the third test function $f_{2}\left(  x,y\right)  =\cos\left(
10x+20y\right)  $ we obtain the following table:\ %

\begin{gather*}
\text{\textsc{Table} 6}\\%
\begin{tabular}
[c]{|l|l|l|l|l|}\hline
N%
$\backslash$%
M & \textbf{9} & \textbf{25} & \textbf{63} & \textbf{83}\\\hline
\textbf{10} & -0.190440454101284 & 0.0936727156130806 & 0.105393884431863 &
0.105393884431863\\\hline
\textbf{15} & -0.120671303989885 & 0.154856209500921 & 0.167165382069484 &
0.167165382069483\\\hline
\textbf{25} & -0.0641280979279670 & 0.207320729809817 & 0.219935669864855 &
0.219935669864855\\\hline
\textbf{35} & -0.0400423613447862 & 0.230317338211144 & 0.243017030992702 &
0.243017030992701\\\hline
\textbf{50} & -0.0220967384451548 & 0.247690217929060 & 0.260435021528943 &
0.260435021528943\\\hline
Error &  &  &  & \\\hline
N%
$\backslash$%
M & \textbf{9} & \textbf{25} & \textbf{63} & \textbf{83}\\\hline
\textbf{10} & 0.49175144943649 & 0.207638279722134 & 0.195917110903352 &
0.195917110903352\\\hline
\textbf{15} & 0.421982299325100 & 0.146454785834294 & 0.134145613265731 &
0.134145613265732\\\hline
\textbf{25} & 0.365439093263182 & 0.0939902655253980 & 0.0813753254703597 &
0.0813753254703599\\\hline
\textbf{35} & 0.341353356680001 & 0.0709936571240705 & 0.0582939643425133 &
0.0582939643425137\\\hline
\textbf{50} & 0.323407733780370 & 0.0536207774061552 & 0.0408759738062723 &
0.0408759738062725\\\hline
\end{tabular}
\\
\text{\textbf{Midpoint cubature} for }f_{2}w^{\left(  1\right)  }=\cos\left(
10x+20y\right)  \left(  \frac{1}{r}+\cos\varphi\right)  .
\end{gather*}

Finally we consider the differentiable test function $f_{3}\left(  x,y\right)
=r^{5/2}.$ We obtain the following results%

\begin{gather*}
\text{\textsc{Table} 7}\\%
\begin{tabular}
[c]{|l|l|l|}\hline
N=M & $I_{N,N}^{\text{mid }}\left(  f_{3}w^{\left(  1\right)  }\right)  $ &
Error\\\hline
\textbf{5} & $1.790188901210552$ & $0.005\,006\,901$\\\hline
\textbf{10} & $1.793908488488327$ & $0.001\,287\,313$\\\hline
\textbf{20} & $1.794870559302513$ & $0.000\,325\,242$\\\hline
\textbf{100} & $1.795182719690259$ & $0.000\,013\,082$\\\hline
\textbf{200} & $1.795192530240442$ & $0.000\,003\,271$\\\hline
\end{tabular}
\\
\text{\textbf{Midpoint cubature for} }f_{3}w^{\left(  1\right)  }%
=r^{5/2}\left(  \frac{1}{r}+\cos\varphi\right)  \text{ }\\
\text{\textbf{True value is} }\approx1.\,795\,195\,802\,051\,31
\end{gather*}
If we compare the $2\times9\times35=630$ valuations with precision $10^{-6}$
in Table 3 with the present $10000$ evaluations with precision $10^{-4},$ we
see how much better is our cubature formula, even if the integrand is a
$C^{1}$ function.

\subsection{Comparison with the generalized Peirce Rule}

S. De and K.J. Bathe discussed in \cite{DeBa00} (see also see \cite[p.
65]{MFPG07}) the following rule: For given $N,$ let $\rho_{j}$, $j=1,...,N$ be
the nodes of the Gauss quadrature $G_{N}$ on $\left[  0,R^{2}\right]  $ with
corresponding weights $w_{j},$further $\alpha$ be a real parameter and $M$ a
natural number: then
\begin{equation}
I_{N,M}^{\text{Peirce,}\alpha}\left(  f\right)  :=\frac{\pi}{M}%
{\displaystyle\sum_{j=1}^{N}}
w_{j}%
{\displaystyle\sum_{s=1}^{M}}
f\left(  \sqrt{\rho_{j}}\cos\frac{2\pi\left(  s+\alpha\right)  }{M},\sqrt
{\rho_{j}}\sin\frac{2\pi\left(  s+\alpha\right)  }{M}\right)
\label{modPeirce}%
\end{equation}
is called the generalized Peirce rule. The rule of Peirce \cite{Peir57} is
obtained by setting $N=m+1,$ $M=4m+4,$ and $\alpha=0.$

\begin{remark}
Let us note that the Discrete Polyharmonic Cubature for the constant weight
function $w\equiv1$ (hence $K=0$ in (\ref{DiscrPolyhCubature})) is identical
with formula (\ref{modPeirce}) with $\alpha=0.$ Thus the numerical evaluation
of this formula can be carried out with our programme.
\end{remark}

The next table gives the results for the rule of Peirce $I_{N,M}%
^{\text{Peirce,0}}\left(  f_{0}w^{\left(  1\right)  }\right)  $ for the first
test function:%

\begin{gather*}
\text{\textsc{Table} 8}\\%
\begin{tabular}
[c]{|l|l|l|l|l|}\hline
N/M & \textbf{9} & \textbf{25} & \textbf{63} & \textbf{83}\\\hline
\textbf{10} & 6.49387212 & 6.49387212 & 6.49387212 & 6.49387212\\\hline
\textbf{15} & 6.577936813 & 6.577936813 & 6.577936813 & 6.577936813\\\hline
\textbf{25} & 6.647152541 & 6.647152541 & 6.647152541 & 6.647152541\\\hline
\textbf{35} & 6.677370918 & 6.677370918 & 6.677370918 & 6.677370918\\\hline
\textbf{50} & 6.700258414 & 6.700258414 & 6.700258414 & 6.700258414\\\hline
\end{tabular}
\\
\text{\textbf{Peirce Cubature rule for} }f_{0}w^{\left(  1\right)  }=\left(
1+x^{4}+y^{3}\right)  \left(  \frac{1}{r}+\cos\varphi\right) \\
\text{\textbf{True value is }}\frac{43}{20}\pi\approx6.754\,\allowbreak
424\,\allowbreak205\,\allowbreak218\,\allowbreak060
\end{gather*}
Even in the case $N=25,$ $M=25,$ with the number of evaluation points equal to
$NM=25\cdot25=625,$ we have an error of
\[
6.754424205218060-6.647152541=0.107\,271\,664\,218\,06
\]
In case of $NM=50\cdot83=4150$ evaluation points (which is about $6$ times
bigger than $625$ ) the error is only about twice smaller:
\[
6.754424205218060-6.700258414=0.054\,165\,791\,\allowbreak218\,06
\]

The experiments with the test functions $f_{1}$, $f_{2}$ and $f_{3}$ have
shown a similar behaviour.

\section{Experimental results for the weight function $w^{\left(  2\right)
}\left(  x,y\right)  $\label{Sexperiments2}}

In this section we discuss a weight function which is of quite different
nature compared to the first weight function. The function
\begin{equation}
w^{\left(  2\right)  }\left(  re^{i\varphi}\right)  :=\left\vert y\right\vert
=\left\vert r\sin\varphi\right\vert \label{w2}%
\end{equation}
is simple in the sense that it is homogeneous in the variable $r.$ However it
has an infinite Fourier series since
\[
\left\vert \sin\varphi\right\vert =\frac{2}{\pi}-\frac{4}{\pi}\sum
_{k=1}^{\infty}\frac{\cos\left(  2k\varphi\right)  }{4k^{2}-1}.
\]
The weight function $w^{\left(  2\right)  }$ is \emph{pseudo-definite} since
its orthonormalized Fourier coefficients have a definite sign:
\[
w_{\left(  0,1\right)  }\left(  r\right)  =\frac{2\sqrt{2}}{\sqrt{\pi}}r\text{
and }w_{\left(  2k,1\right)  }\left(  r\right)  =-\frac{4}{\sqrt{\pi}}\frac
{1}{4k^{2}-1}r\text{ for }k\geq1.
\]
We recall our main integration formula (\ref{eqint})
\begin{equation}
I_{w}\left(  f\right)  =\int_{0}^{2\pi}\int_{0}^{R}f\left(  re^{i\varphi
}\right)  w\left(  re^{i\varphi}\right)  rdrd\varphi=%
{\displaystyle\sum_{k=0}^{\infty}}
{\displaystyle\sum_{\ell=1}^{a_{k}}}
\int_{0}^{R}f_{\left(  k,\ell\right)  }\left(  r\right)  w_{\left(
k,\ell\right)  }\left(  r\right)  rdr. \label{eqlast1}%
\end{equation}

\subsection{Results for the discrete polyharmonic cubature}

Further we consider the test function
\begin{equation}
f_{4}\left(  x,y\right)  =30x^{12}. \label{fun4w}%
\end{equation}
Note that
\[
I_{w^{\left(  2\right)  }}\left(  f_{4}\right)  =I_{1}\left(  30x^{12}%
\left\vert y\right\vert \right)  =\int_{0}^{1}\int_{0}^{2\pi}30r^{12}\cos
^{12}\left(  \varphi\right)  \left\vert \sin\varphi\right\vert r^{2}%
drd\varphi=\frac{8}{13}\approx0.615\,384\,615\,384\,610.
\]
Since $f_{4}\left(  x,y\right)  $ is a polynomial of degree $12$ the Fourier
coefficients satisfy $f_{4,\left(  k,\ell\right)  }\left(  r\right)  =0$ for
$k>13,$ hence, we may take $K=12.$ For $M>24$ and $N>12$ the polynomial
$f\left(  x,y\right)  $ will be reproduced, i.e. holds
\[
I_{\left(  N,M,12\right)  }^{\text{poly }}\left(  f\right)  =I_{w^{\left(
2\right)  }}\left(  f\right)  =I_{1}\left(  fw^{\left(  2\right)  }\right)  .
\]
This can be seen in the following table:%
\begin{gather*}
\text{\textsc{Table} 9}\\%
\begin{tabular}
[c]{|r|rrrr|}\hline
N%
$\backslash$%
M & \multicolumn{1}{|r|}{\textbf{9}} & \multicolumn{1}{r|}{\textbf{25}} &
\multicolumn{1}{r|}{\textbf{63}} & \textbf{83}\\\hline
\textbf{10} & \multicolumn{1}{|r|}{0.5609353695139790} &
\multicolumn{1}{r|}{0.6153846153846160} &
\multicolumn{1}{r|}{0.6153846153846160} & 0.6153846153846160\\\hline
\textbf{15} & \multicolumn{1}{|r|}{0.5609353656165750} &
\multicolumn{1}{r|}{0.6153846153846150} &
\multicolumn{1}{r|}{0.6153846153846150} & 0.6153846153846150\\\hline
\textbf{25} & \multicolumn{1}{|r|}{0.5609353655541600} &
\multicolumn{1}{r|}{0.6153846153846170} &
\multicolumn{1}{r|}{0.6153846153846170} & 0.6153846153846170\\\hline
\textbf{35} & \multicolumn{1}{|r|}{0.5609353655539850} &
\multicolumn{1}{r|}{0.6153846153846130} &
\multicolumn{1}{r|}{0.6153846153846140} & 0.6153846153846140\\\hline
\textbf{50} & \multicolumn{1}{|r|}{0.5609353655539860} &
\multicolumn{1}{r|}{0.6153846153846170} &
\multicolumn{1}{r|}{0.6153846153846170} & 0.6153846153846170\\\hline
Error & \multicolumn{1}{|r|}{} & \multicolumn{1}{r|}{} & \multicolumn{1}{r|}{}
& \\\hline
N%
$\backslash$%
M & \multicolumn{1}{|r|}{\textbf{9}} & \multicolumn{1}{r|}{\textbf{25}} &
\multicolumn{1}{r|}{\textbf{63}} & \textbf{83}\\\hline
\textbf{10} & \multicolumn{1}{|r|}{0.0544492458706369} &
\multicolumn{1}{r|}{0.0000000000000000} &
\multicolumn{1}{r|}{0.0000000000000000} & 0.0000000000000000\\\hline
\textbf{15} & \multicolumn{1}{|r|}{0.0544492497680410} &
\multicolumn{1}{r|}{0.0000000000000010} &
\multicolumn{1}{r|}{0.0000000000000010} & 0.0000000000000010\\\hline
\textbf{25} & \multicolumn{1}{|r|}{0.0544492498304560} &
\multicolumn{1}{r|}{0.0000000000000010} &
\multicolumn{1}{r|}{0.0000000000000010} & 0.0000000000000010\\\hline
\textbf{35} & \multicolumn{1}{|r|}{0.0544492498306309} &
\multicolumn{1}{r|}{0.0000000000000030} &
\multicolumn{1}{r|}{0.0000000000000020} & 0.0000000000000020\\\hline
\textbf{50} & \multicolumn{1}{|r|}{0.0544492498306299} &
\multicolumn{1}{r|}{0.0000000000000010} &
\multicolumn{1}{r|}{0.0000000000000010} & 0.0000000000000010\\\hline
\end{tabular}
\\
\text{\textbf{Discrete Polyharmonic cubature} for }K=12\text{ and }%
f_{4}w^{\left(  2\right)  }=30x^{12}\left\vert y\right\vert \\
\text{\textbf{True value} is }\approx0.615\,384\,615\,\allowbreak384\,616
\end{gather*}

Next we consider the test function
\begin{equation}
f_{5}\left(  x,y\right)  :=\left\vert y\right\vert . \label{fun5w}%
\end{equation}
Since it is not smooth we should not expect to obtain too good results for the
discrete polyharmonic cubature. The exact value of the integral can be
computed:
\begin{align*}
I_{1}\left(  \left\vert y\right\vert \left\vert y\right\vert \right)   &
=\int_{0}^{1}\int_{0}^{2\pi}\left\vert r\sin\varphi\right\vert ^{2}d\varphi
rdr=\int_{0}^{1}r^{3}dr\cdot\int_{0}^{2\pi}\left\vert \sin\varphi\right\vert
^{2}d\varphi\\
&  =\frac{1}{4}\pi\approx0.785\,398\,163\,397\,448\,
\end{align*}
We took $K=12,$ as in the last experiment, and we obtained a table of results
where the error is almost the same for all $N$ and $M$, and is around
$10^{-3}.$ If we take $K=22$ the results are much better:
\begin{gather*}
\text{\textsc{Table} 10}\\%
\begin{tabular}
[c]{|r|rrrr|}\hline
N%
$\backslash$%
M & \multicolumn{1}{|r|}{\textbf{9}} & \multicolumn{1}{r|}{\textbf{25}} &
\multicolumn{1}{r|}{\textbf{63}} & \textbf{83}\\\hline
\textbf{10} & \multicolumn{1}{|r|}{0.785206660} &
\multicolumn{1}{r|}{0.785352337} & \multicolumn{1}{r|}{0.785367124} &
0.785369362\\\hline
\textbf{15} & \multicolumn{1}{|r|}{0.785208297} &
\multicolumn{1}{r|}{0.785358970} & \multicolumn{1}{r|}{0.785373081} &
0.785375274\\\hline
\textbf{25} & \multicolumn{1}{|r|}{0.785208235} &
\multicolumn{1}{r|}{0.785361119} & \multicolumn{1}{r|}{0.785374994} &
0.785377171\\\hline
\textbf{35} & \multicolumn{1}{|r|}{0.785208149} &
\multicolumn{1}{r|}{0.785361440} & \multicolumn{1}{r|}{0.785375276} &
0.785377452\\\hline
\textbf{50} & \multicolumn{1}{|r|}{0.785208109} &
\multicolumn{1}{r|}{0.785361541} & \multicolumn{1}{r|}{0.785375364} &
0.785377539\\\hline
Error & \multicolumn{1}{|r|}{} & \multicolumn{1}{r|}{} & \multicolumn{1}{r|}{}
& \\\hline
N%
$\backslash$%
M & \multicolumn{1}{|r|}{\textbf{9}} & \multicolumn{1}{r|}{\textbf{25}} &
\multicolumn{1}{r|}{\textbf{63}} & \textbf{83}\\\hline
\textbf{10} & \multicolumn{1}{|r|}{0.0001915037} &
\multicolumn{1}{r|}{0.0000458267} & \multicolumn{1}{r|}{0.0000310393} &
0.0000288009\\\hline
\textbf{15} & \multicolumn{1}{|r|}{0.0001898666} &
\multicolumn{1}{r|}{0.0000391939} & \multicolumn{1}{r|}{0.0000250824} &
0.0000228890\\\hline
\textbf{25} & \multicolumn{1}{|r|}{0.0001899281} &
\multicolumn{1}{r|}{0.0000370443} & \multicolumn{1}{r|}{0.0000231697} &
0.0000209919\\\hline
\textbf{35} & \multicolumn{1}{|r|}{0.0001900142} &
\multicolumn{1}{r|}{0.0000367231} & \multicolumn{1}{r|}{0.0000228870} &
0.0000207116\\\hline
\textbf{50} & \multicolumn{1}{|r|}{0.0001900544} &
\multicolumn{1}{r|}{0.0000366227} & \multicolumn{1}{r|}{0.0000227993} &
0.0000206248\\\hline
\end{tabular}
\\
\text{\textbf{Discrete Polyharmonic cubature} for }K=22\text{ and }%
f_{5}w^{\left(  2\right)  }=\left\vert y\right\vert ^{2}\text{ }\\
\text{\textbf{True value} is }\approx0.785\,398\,163\,397\,448
\end{gather*}

Finally, we look again at the oscillating test function (\ref{fun3}), namely,
$f_{2}\left(  x,y\right)  =\cos\left(  10x+20y\right)  .$Using the expansion%
\[
C^{a,b}\left(  x,y\right)  :=\cos\left(  ax+by\right)  =%
{\displaystyle\sum_{n=0}^{\infty}}
\frac{\left(  -1\right)  ^{n}}{\left(  2n\right)  !}\left(  ax+by\right)
^{2n}%
\]
one can prove that the Fourier coefficients $C_{\left(  2k+1,\ell\right)
}^{a,b}\left(  r\right)  $ are zero and
\[
C_{\left(  2k,\ell\right)  }^{a,b}\left(  r\right)  =2\pi Y_{2k,\ell}\left(
\varphi_{a,b}\right)  \left(  -1\right)  ^{k}J_{2k}\left(  r\sqrt{\left(
a^{2}+b^{2}\right)  }\right)
\]
where $J_{2k}$ are the Bessel functions, see formula (\ref{bessel}), and
$\varphi_{a,b}$ is the angle given by $\tan\varphi_{a,b}=\frac{b}{a}.$ It
follows that
\begin{equation}
I_{w^{\left(  2\right)  }}\left(  f\right)  =4\int_{0}^{1}J_{0}\left(
r\sqrt{a^{2}+b^{2}}\right)  r^{2}dr+%
{\displaystyle\sum_{k=1}^{\infty}}
8\frac{\left(  -1\right)  ^{k+1}\cos\left(  2k\varphi_{a,b}\right)  }%
{4k^{2}-1}\int_{0}^{1}J_{2k}\left(  r\sqrt{a^{2}+b^{2}}\right)  r^{2}dr.
\label{Iwf}%
\end{equation}
We may estimate the error using this expression.

In formula (\ref{DiscrPolyhCubature+}), for the \textbf{Discrete Polyharmonic
Cubature, }we have taken $K=22.$ We obtain the following results:
\begin{gather*}
\text{\textsc{Table }11}\\%
\begin{tabular}
[c]{|r|rrrr|}\hline
N%
$\backslash$%
M & \multicolumn{1}{|r|}{\textbf{9}} & \multicolumn{1}{r|}{\textbf{25}} &
\multicolumn{1}{r|}{\textbf{63}} & \textbf{83}\\\hline
\textbf{10} & \multicolumn{1}{|r|}{0.096718846427391} &
\multicolumn{1}{r|}{0.014472433304185} &
\multicolumn{1}{r|}{0.014477271351135} & 0.014477271351135\\\hline
\textbf{15} & \multicolumn{1}{|r|}{0.096642162991824} &
\multicolumn{1}{r|}{0.014472441635349} &
\multicolumn{1}{r|}{0.014477279682299} & 0.014477279682299\\\hline
\textbf{25} & \multicolumn{1}{|r|}{0.096593708566055} &
\multicolumn{1}{r|}{0.014472441635349} &
\multicolumn{1}{r|}{0.014477279682299} & 0.014477279682299\\\hline
\textbf{35} & \multicolumn{1}{|r|}{0.096592741482745} &
\multicolumn{1}{r|}{0.014472441635350} &
\multicolumn{1}{r|}{0.014477279682299} & 0.014477279682299\\\hline
\textbf{50} & \multicolumn{1}{|r|}{0.096597855764522} &
\multicolumn{1}{r|}{0.014472441635349} &
\multicolumn{1}{r|}{0.014477279682299} & 0.014477279682299\\\hline
Error & \multicolumn{1}{|r|}{} & \multicolumn{1}{r|}{} & \multicolumn{1}{r|}{}
& \\\hline
N%
$\backslash$%
M & \multicolumn{1}{|r|}{\textbf{9}} & \multicolumn{1}{r|}{\textbf{25}} &
\multicolumn{1}{r|}{\textbf{63}} & \textbf{83}\\\hline
\textbf{10} & \multicolumn{1}{|r|}{0.082241566745092} &
\multicolumn{1}{r|}{0.000004846378115} &
\multicolumn{1}{r|}{0.000000008331165} & 0.000000008331165\\\hline
\textbf{15} & \multicolumn{1}{|r|}{0.082164883309524} &
\multicolumn{1}{r|}{0.000004838046951} &
\multicolumn{1}{r|}{0.000000000000001} & 0.000000000000001\\\hline
\textbf{25} & \multicolumn{1}{|r|}{0.082116428883756} &
\multicolumn{1}{r|}{0.000004838046950} &
\multicolumn{1}{r|}{0.000000000000000} & 0.000000000000000\\\hline
\textbf{35} & \multicolumn{1}{|r|}{0.082115461800446} &
\multicolumn{1}{r|}{0.000004838046950} &
\multicolumn{1}{r|}{0.000000000000000} & 0.000000000000000\\\hline
\textbf{50} & \multicolumn{1}{|r|}{0.082120576082222} &
\multicolumn{1}{r|}{0.000004838046950} &
\multicolumn{1}{r|}{0.000000000000000} & 0.000000000000000\\\hline
\end{tabular}
\\
\text{\textbf{Discrete Polyharmonic cubature} for }K=22\text{ and }%
f_{2}w^{\left(  2\right)  }=\cos\left(  10r\cos\varphi+20r\sin\varphi\right)
\left\vert r\sin\varphi\right\vert \\
\text{\textbf{True value} is }\approx0.0144772796822995
\end{gather*}

\subsection{Comparison with piece-wise midpoint rule}

We want to give briefly the results for the piece-wise midpoint rule for the
test function $f_{4}\left(  x,y\right)  =30x^{12}$ with the weight $w^{\left(
2\right)  }\left(  x,y\right)  =\left\vert y\right\vert .$
\begin{align*}
I_{5,5}^{\text{mid}}\left(  30x^{12}\left\vert y\right\vert \right)   &
=0.173\,359\,053\,\allowbreak300\,102\,\\
I_{10,10}^{\text{mid}}\left(  30x^{12}\left\vert y\right\vert \right)   &
=0.795\,909\,972\,\allowbreak453\,979\\
I_{20,20}^{\text{mid}}\left(  30x^{12}\left\vert y\right\vert \right)   &
=0.640\,355\,591\,\allowbreak783\,074\\
I_{40,40}^{\text{mid}}\left(  30x^{12}\left\vert y\right\vert \right)   &
=0.620\,920\,690\,\allowbreak132\,442\\
I_{100,100}^{\text{mid}}\left(  30x^{12}\left\vert y\right\vert \right)   &
=0.616\,243\,873\,\allowbreak682\,115\\
I_{200,200}^{\text{mid}}\left(  30x^{12}\left\vert y\right\vert \right)   &
=0.615\,598\,519\,\allowbreak244\,782\\
I_{500,500}^{\text{mid}}\left(  30x^{12}\left\vert y\right\vert \right)   &
=0.615\,418\,799\,\allowbreak448\,66\\
\text{\textbf{True value} is}\text{ }  &  \approx0.615\,384\,615\,\allowbreak
384\,616
\end{align*}
We see that even with $500\times500$ evaluation points the error is around
$0.0001.$ On the other hand, for the test function $f_{5}\left(  x,y\right)
=\left\vert y\right\vert $ the piecewise midpoint rule is very good, since the
integrand is just $\left\vert y\right\vert ^{2}$ but we do not have exactness.

We omit a discussion of the Peirce rules since the results are very similar to
the case of the midpoint cubature rule.

\section{Aspects of the Numerical Implementation \label{Spractical}}

The discrete polyharmonic cubature (\ref{DiscrPolyhCubature}) is defined for
pseudo-definite weight functions $w\left(  re^{i\varphi}\right)  $ on the
disk. It is important to determine numerically the nodes and weights for the
quadrature of degree $N$ for the univariate integrals
\[
\frac{1}{2}\int_{0}^{R^{2}}P\left(  \rho\right)  \rho^{k/2}w_{\left(
k,\ell\right)  }\left(  \sqrt{\rho}\right)  d\rho
\]
with high accuracy. The nodes are the zeros of the orthogonal polynomial
$P_{n}\left(  \rho\right)  $ of degree $N$ with respect to the measure
$\rho^{k/2}w_{\left(  k,\ell\right)  }\left(  \sqrt{\rho}\right)  .$ Thus any
implementation of our algorithm depends on reliable and stable software for
finding the nodes and weights of the corresponding quadrature. We have
judiciously chosen the weights $w^{\left(  1\right)  }$ and $w^{\left(
2\right)  }$ such that
\[
w_{\left(  k,\ell\right)  }\left(  \sqrt{\rho}\right)  =C_{\left(  k,l\right)
}r^{\alpha\left(  k,\ell\right)  }\left(  1-r\right)  ^{\beta\left(
k,\ell\right)  }%
\]
are weight functions of Jacobi type, i.e. of the form $x^{\alpha}\left(
1-x\right)  ^{\beta}$ with $\alpha,\beta\geq-1$ in the interval $\left[
0,1\right]  $. For these weight functions fast and highly accurate programmes
for finding the nodes and weights are available (see e.g.
\cite{GautschiBOOKorthogonal} or the website:
https://www.cs.purdue.edu/archives/2002/wxg/codes) which we have used in our
Matlab programs. A more general situation, namely when the Fourier
coefficients $w_{\left(  k,\ell\right)  }\left(  \sqrt{\rho}\right)  $ are
linear combinations of Jacobi type weight functions, is easy to handle. In the
case that $w_{\left(  k,\ell\right)  }\left(  \sqrt{\rho}\right)  $ is a
general pseudo-definite function the user has to search or develop numerical
procedures for finding the nodes and coefficients with high accuracy. Another
idea, which we did not pursue in detail, is to approximate $w_{\left(
k,\ell\right)  }\left(  \sqrt{\rho}\right)  $ by means of a Bernstein
polynomial $\sum_{j=0}^{d}\gamma_{j}\left(  1-\rho\right)  ^{j}\rho^{d-j}$ of
some degree $d,$ which is then a linear combination of Jacobi weights.
However, a problem might be the approximation rate for $d\longrightarrow
\infty$, which is notoriously not very good.

In our Matlab program we used the readily implemented Matlab function for the
Fast Fourier Transform: then the Fourier series is in the form
\[
w\left(  re^{i\varphi}\right)  =%
{\displaystyle\sum_{k=-\infty}^{\infty}}
w_{k}\left(  r\right)  e^{ik\varphi}%
\]
with the usual relations $w_{0}\left(  r\right)  =\frac{1}{2}w_{\left(
0,1\right)  }\left(  r\right)  $ and
\[
w_{\left(  k,1\right)  }\left(  r\right)  =\frac{1}{2}\left(  w_{k}\left(
r\right)  +\overline{w_{k}\left(  r\right)  }\right)  \text{ and }w_{\left(
k,2\right)  }\left(  r\right)  =\frac{1}{2i}\left(  w_{k}\left(  r\right)
-\overline{w_{k}\left(  r\right)  }\right)  .
\]
Here $w_{k}\left(  r\right)  :=\frac{1}{2\pi}\int_{0}^{2\pi}w\left(
re^{i\varphi}\right)  e^{-ik\varphi}d\varphi$. For the computation of the
approximation $f_{\left(  k,\ell\right)  }^{\left(  M\right)  }$ to the
Fourier coefficients $f_{\left(  k,\ell\right)  }\left(  r\right)  $ of the
integration function $f\left(  x,y\right)  $ we use the Fourier series
representation
\[
f\left(  re^{i\varphi}\right)  =\sum_{k=-\infty}^{\infty}f_{k}\left(
r\right)  e^{ik\varphi}\text{, where }f_{k}\left(  r\right)  :=\frac{1}{2\pi
}\int_{0}^{2\pi}f\left(  re^{i\varphi}\right)  e^{-ik\varphi}d\varphi.
\]
Hence, the integral (\ref{eqcentral2}) can be written in the form
\begin{equation}
I\left(  f\right)  =%
{\displaystyle\sum_{k=-\infty}^{\infty}}
\int_{0}^{R}2\pi f_{k}\left(  r\right)  w_{-k}\left(  r\right)  rdr=%
{\displaystyle\sum_{k=-\infty}^{\infty}}
\int_{0}^{R}\pi f_{k}\left(  \sqrt{\rho}\right)  w_{-k}\left(  \sqrt{\rho
}\right)  d\rho.\label{eqcentral}%
\end{equation}
In our program we choose always odd $M,$ and the discrete Fourier
approximation to $f_{k}\left(  r\right)  $ given by
\[
f_{k}^{\left(  M\right)  }\left(  r\right)  :=f_{\left(  k,1\right)
}^{\left(  M\right)  }\left(  r\right)  -if_{\left(  k,2\right)  }^{\left(
M\right)  }\left(  r\right)  =\frac{1}{M}%
{\displaystyle\sum_{s=1}^{M}}
f\left(  re^{2\pi i\frac{s}{M}}\right)  e^{-2\pi i\frac{ks}{M}}%
\]
which is the link to the Fast Fourier transform. By subdividing the interval
$\left[  0,2\pi\right]  $ into $M$ subintervals we obtain an approximation
\[
f_{k}^{M}\left(  r\right)  :=\frac{1}{M}%
{\displaystyle\sum_{s=1}^{M}}
f\left(  re^{2\pi is/M}\right)  e^{-2\pi iks/M}%
\]
of $f_{k}\left(  r\right)  $ which is just the Discrete Fourier transform
(DFT) for the data points $f\left(  re^{2\pi is/M}\right)  $ for $s=1,...,M.$

\section{Concluding Remarks}

The discrete polyharmonic cubature formula $I_{\left(  N,M,K\right)
}^{\text{poly}}\left(  f\right)  $ provides excellent numerical results for
integrating functions on the disk in the plane with respect to a weight
function $w.$ A possible drawback for applications might be the high number of
evaluations points needed in the formula given by $\left(  2K-1\right)  \cdot
N\cdot M$ evaluation points. Since the evaluation of a function value
$f\left(  x\right)  $ might be very costly, it is a natural question whether
one could modify the formula so that we would need less function evaluations.

In a forthcoming paper \cite{kounchevRenderHybrid} we introduce a new cubature
formula related to the above formula (\ref{DiscrPolyhCubature+}) with a high
degree of "approximative exactness", which uses a spline approximation in
direction $r$ of the function $f_{\left(  k,\ell\right)  }^{\left(  M\right)
}\left(  r\right)  $. The point evaluations for this formula are on a regular
grid in both directions $\varphi$ and $r$, and the number of knots is
$N_{1}\times M$, where $N_{1}$ is the number of spline knots in direction $r.$
We call this a hybrid cubature since we combine spline methods with a cubature
formula. In mathematical terms the hybrid formula is of the form:
\begin{equation}
I_{\left(  N,M,K\right)  }^{\text{poly}}\left(  f\right)  \approx\frac{1}%
{2}\sum_{k=0}^{K}\sum_{\ell=1}^{a_{k}}\sum_{j=1}^{N}SPL\left[  \left\{
f_{\left(  k,\ell\right)  }^{\left(  M\right)  }\left(  R_{m}\right)
\right\}  _{m=0}^{N_{1}}\right]  \left(  \sqrt{t_{j,\left(  k,\ell\right)  }%
}\right)  \times t_{j,\left(  k,\ell\right)  }^{-\frac{k}{2}}\lambda
_{j,\left(  k,\ell\right)  }, \label{r5}%
\end{equation}
where $SPL\left[  \left\{  f_{\left(  k,\ell\right)  }^{\left(  M\right)
}\left(  R_{m}\right)  \right\}  _{m=0}^{N_{1}}\right]  \left(  t\right)  $ is
a univariate spline interpolation function with nodes $\left\{  R_{m}\right\}
_{m=0}^{N_{1}}$ for the data $\left\{  f_{\left(  k,\ell\right)  }^{\left(
M\right)  }\left(  R_{m}\right)  \right\}  _{m=0}^{N_{1}}$.

\begin{acknowledgement}
Both authors acknowledge the partial support by the Bulgarian NSF Grant
I02/19, 2015. The first named author acknowledges partial support by the
Humboldt Foundation.
\end{acknowledgement}

Author 1 affiliations:\ Institute of Mathematics and Informatics, Bulgarian
Academy of Science, Acad. G. Bonchev st., bl. 8, 1113 Sofia, Bulgaria, and
IZKS-University of Bonn.

Author 2 affiliation: School of Mathematics and Statistics, University College
Dublin, Belfield 4, Dublin, Ireland.


\begin{thebibliography}{99}                                                                                               %


\bibitem {Ahli62}A.C. Ahlin, \emph{On Error bounds for Gaussian cubature},
SIAM Review 4 (1962), 25--39.

\bibitem {AlCo58}J. Albrecht, L. Collatz, \emph{Zur numerischen Auswertung
mehrdimensionaler Integrale,} ZAMM 38 (1958), 1--15.

\bibitem {Albr60}J. Albrecht, \emph{Formeln zur numerischen Auswertung
\"{u}ber Kreisbereiche}, ZAMM 40 (1960), 514--517.

\bibitem {AlEn76}J. Albrecht, H. Engels, \emph{Zur numerischen Integration}
\emph{\"{u}ber Kreisbereiche, }Constructive Theory of Functions of Several
Variables, Lecture Notes in Mathematics Volume 571, 1977, pp 1-5.

\bibitem {AAR99}G.E. Andrews, R. Askey, R. Roy, \emph{Special functions.}
Encyclopedia of Mathematics and its Applications, 71. Cambridge University
Press, Cambridge, 1999.

\bibitem {Appe26}P. Appell, J. Kamp\'{e} de F\'{e}riet, \emph{Fonction
hyperg\'{e}om\'{e}triques et hypersph\'{e}riques; Polynomes d'Hermite,}
Gauthier-Villar, Paris, 1926.

\bibitem {Appe1890}P. Appel, \emph{Sur une Classe de Polynomes \`{a} Deux
Variables et le Calcul Approchedes Integrales Doubles,} Ann. de la Facult\'{e}
des Sciences de Toulouse 4 (1890), 1--20.

\bibitem {Baud}N. Baudin, \emph{The Integration Test Problems Toolbox,}
Internet 2010.

\bibitem {Bour1898}H. Bourget, \emph{Sur une extension de la m\'{e}thode de
quadrature de Gauss,} Hebdomadaires des s\'{e}ances de l'Acad\'{e}mie des
Sciences, Comptes Rendus, v. 126, 1898, p. 634--636.

\bibitem {Cool03}R. Cools, \emph{An Encyclopaedia of Cubature Formulas,} J.
Complexity, 19 (2003), 445--453.

\bibitem {CoHa87}R. Cools, A. Haegemans, \emph{Automatic computation of knots
and weights of cubature formulae for circular symmetric planar regions, }J.
Comput. Appl. Math. 20 (1987) 153--158.

\bibitem {CoKi00}R. Cools, K. Kim, \emph{A survey of known and new cubature
formulas for the unit disk}, Korean J. Comput. Appl. Math. 7 (3) (2000) 477--485.

\bibitem {dahlquist}G. Dahlquist, E. Bj\"{o}rck, \emph{Numerical methods in
scientific computing}, Society for Industrial and Applied Mathematics,
Philadelphia, 2008.

\bibitem {Dai-xu}F. Dai, Y. Xu, \emph{Approximation theory and harmonic
analysis on spheres and balls}, Springer, Berlin, 2013.

\bibitem {Davis}P.J. Davis, \emph{Interpolation and Approximation}. Dover
Publications Inc., New York, 1975.

\bibitem {Davi59}P.J. Davis, \emph{On the numerical integration of periodic
analytic functions,} In "On Numerical Approximation" (R.E. Langer), 45--60.
The University of Wisconsin Press, Madison.

\bibitem {DaRa75}P. Davis, P. Rabinowitz, \emph{Methods of Numerical
Integration.} Second edition. Computer Science and Applied Mathematics.
Academic Press, Inc., Orlando, FL, 1984.

\bibitem {DeBa00}S. De, K.J. Bathe, \emph{The method of finite spheres,
}Comput. Mech. 25 (2000), 329--345.

\bibitem {DeBa01}S. De, K.J. Bathe, \emph{The method of finite spheres with
improved numerical integration,} Comput. Struct. 79 (2001) 2183--2196.

\bibitem {Engels}H. Engels, \emph{Numerical Quadrature and Cubature,} Academic
Press, London 1977.

\bibitem {Epst05}C.L. Epstein, \emph{How well does the Finite Fourier
Transform approximate the Fourier transform,} Comm. Pure Appl. Math. 58
(2005), 1--15.

\bibitem {GautschiBOOKorthogonal}W. Gautschi, Orthogonal Polynomials,
Computation and Approximation, Oxford University Press, 2004.

\bibitem {Gaut97}W. Gautschi, \emph{Numerical Analysis,} An Introduction,
Birkh\"{a}user Boston 1997.

\bibitem {genz}A. Genz, Testing multidimensional integration routines.
Tools,Methods and Languages for Scientific and Engineering Computation, pages
$81-94,$ 1984.

\bibitem {Haeg76}A. Haegemans, \emph{Circularly symmetrical integration
formulas for two-dimensional circularly symmetrical regions,} BIT 16 (1976), 52--59.

\bibitem {KiSo97}K. Kim, M. Song,\emph{\ Symmetric quadrature formulas over a
unit disk,} Korean J. Comput. Appl. Math. 4 (1) (1997) 179 192.

\bibitem {okbook}O. Kounchev, \emph{Multivariate Polysplines. Applications to
Numerical and Wavelet Analysis,} Academic Press, San Diego, 2001.

\bibitem {kounchevRenderArkiv}O. Kounchev, H. Render, \emph{A moment problem
for pseudo-positive definite functionals,} Arkiv f\oe r Matematik, 48 (2010), 97-120.

\bibitem {kounchevRenderKleinDirac}O. Kounchev, H. Render, \emph{Polyharmonic
Hardy spaces on the Complexified Annulus and Error estimates of Cubature
formulas}, Results Math. 62 (2012), 377--403.

\bibitem {Dolo13}O. Kounchev, H. Render, \emph{Error Estimates for
Polyharmonic Cubature Formulas, }Dolomites Research Notes on Approximation,
Vol. 6 (2013), 62--73.

\bibitem {kounchevRenderArxiv}O. Kounchev, H. Render, \emph{Reconsideration of
the multivariate moment problem and a new method for approximating
multivariate integrals,} electronic version at arXiv:math/0509380v1 [math.FA]

\bibitem {kounchevRenderHybrid}O. Kounchev, H. Render, \emph{On a Hybrid
Polyharmonic Cubature Formula on the Disc, using splines.} In preparation.

\bibitem {kreinNudelman}M. Krein, A. Nudelman, \emph{The Markov moment problem
and extremal problems,} Amer. Math. Soc., Providence, R.I., 1977.

\bibitem {krylov}V. Krylov, \emph{Approximate calculation of integrals.
Translated by Arthur H. Stroud,} The Macmillan Co., New York-London, 1962.

\bibitem {MaPi10}A. Mazzia, G. Pino, \emph{Product Gauss quadrature rules vs.
cubature rules in the meshless local Petrov Galerkin method,} J. Complexity 26
(2010), 82--101.

\bibitem {MFPG07}A. Mazzia, M. Ferronato, G. Pini, G. Gambolati, \emph{A
comparison of numerical integration rules for the Meshless Local
Petrov-Galerkin method,} Numer. Algor. 45 (2007) 61--74.

\bibitem {Maxw1877}J.C. Maxwell, \emph{On Approximate Multiple Integration
between Limits and Summation,} Cambridge Phil. Soc., Proc., 3 (1877), 39--47.

\bibitem {Mise36}R. von Mises, \emph{Formules de cubature}, Revue Math. de
l'Union Interbalkanique, Athen 1936, 17--31.

\bibitem {Mise53}R. von Mises, \emph{Numerische Berechnung mehrdimensionaler
Integrale}, ZAMM 34 (1953), 201--210.

\bibitem {Pech06}R. Pecher, \emph{Efficient cubature formulae for MLPG and
related methods,} Int. J. Numer. Methods Eng. 65 (2006) 566 593.

\bibitem {Peir57}W.H. Peirce, \emph{Numerical integration over the planar
annulus,} J. Soc. Indust. Appl. Math. 5 (2) (1957) 66--73.

\bibitem {Rado48}J. Radon, \emph{Zur mechanischen Kubatur}, Monatsh. Math. 52
(1948), 286--300.

\bibitem {schurer}R. Sch\"{u}rer. \emph{Parallel High-dimensional Integration:
Quasi-Monte Carlo versus Adaptive Cubature Rules}. In V. N. Alexandrov, J. J.
Dongarra, B. A. Juliano, R. S. Renner and C. J. K. Tan, editors, Proceedings
of the International Conference on Computational Science -- ICCS 2001, volume
2073 of Lecture Notes in Computer Science, pages 1262--1271. Springer-Verlag, 2001.

\bibitem {sobolev}S.L. Sobolev, \emph{Cubature formulas and modern analysis.
An introduction.} Translated from the 1988 Russian edition. Gordon and Breach
Science Publishers, Montreux, 1992.

\bibitem {sobolev2}S. Sobolev, V. Vaskevich, \emph{The theory of cubature
formulas,} Springer, Berlin, 1997.

\bibitem {stoer}Josef Stoer, R. Bulirsch, \emph{Introduction to Numerical
Analysis,} Springer, Berlin, second edition, $2002.$

\bibitem {stroudBook}A.H. Stroud, \emph{Approximate calculation of multiple
integrals}, Prentice-Hall, Englewood Cliffs, N.J., 1971.

\bibitem {VeCo92}P. Verlinden, R. Cools, \emph{On cubature formulae of degree
4k+1 attaining M\"{o}ller's bound for integrals with circular symmetry,}
Numer. Math. 61 (1992), 395--407.

\bibitem {vioreanuRokhlin}B. Vioreanu, V. Rokhlin, \emph{Spectra of
multiplication operators as a numerical tool,} SIAM J. Sci. Comput., Vol. 36,
No. 1, $2014,$ pp. A267--A288.

\bibitem {xiaoGimbutas}H. Xiao and Z. Gimbutas, \emph{A numerical algorithm
for the construction of efficient quadrature rules in two and higher
dimensions,} Comput. Math. Appl., 59 (2010), pp. 663--676.
\end{thebibliography}
\end{document}